\documentclass{amsart}

\usepackage{amsmath}
\usepackage{amssymb,amsthm}
\usepackage[all]{xy}
\usepackage{graphicx}
\usepackage{mathrsfs}
\usepackage{stmaryrd}
\usepackage[active]{srcltx}
\usepackage{euscript}

\usepackage{mathabx,epsfig}

\RequirePackage{ifpdf}
\ifpdf
  \IfFileExists{pdfsync.sty}{\RequirePackage{pdfsync}}{}
  \RequirePackage[pdftex, 
   colorlinks = true,
   pagebackref,
   bookmarksopen=true]{hyperref}
\else
  \RequirePackage[hypertex]{hyperref}
\fi

\usepackage{tikz}
\tikzset{
	MyPersp/.style={scale=2,x={(0.8cm,0cm)},y={(0cm,0.25cm)},
    z={(0cm,1cm)}},
	MyPoints/.style={fill=white,draw=black,thick}
		}

\usetikzlibrary{decorations.pathmorphing,arrows,calc}

 \definecolor{darkgreen}{HTML}{336633}
 \definecolor{darkred}{HTML}{993333}

\RequirePackage{color}
\definecolor{myred}{rgb}{0.75,0,0}
\definecolor{mygreen}{rgb}{0,0.5,0}
\definecolor{myblue}{rgb}{0,0,0.65}

\newcommand{\R}{\mathbb{R}}

\newcommand{\F}{\mathbb{F}}

\newcommand{\K}{\mathbb{K}}

\newcommand{\bk}{\Bbbk}

\newcommand{\Z}{\mathbb{Z}}

\newcommand{\scA}{\mathscr{A}}

\newcommand{\Db}{D^{\mathrm{b}}}

\newcommand{\Par}{\mathsf{Parity}}

\newcommand{\Gv}{G^{\vee}}

\newcommand{\bX}{\mathbf{X}}

\newcommand{\Gr}{{\EuScript Gr}}
\newcommand{\Fl}{{\EuScript Fl}}

\newcommand{\Perv}{\mathsf{Perv}}

\newcommand{\cF}{\mathcal{F}}
\newcommand{\cE}{\mathcal{E}}
\newcommand{\cG}{\mathcal{G}}

\newcommand{\cH}{\mathcal{H}}

\newcommand{\simto}{\xrightarrow{\sim}}

\makeatletter
\def\lotimes{\@ifnextchar_{\@lotimessub}{\@lotimesnosub}}
\def\@lotimessub_#1{\mathchoice{\mathbin{\mathop{\otimes}^L}_{#1}}%
  {\otimes^L_{#1}}{\otimes^L_{#1}}{\otimes^L_{#1}}}
\def\@lotimesnosub{\mathbin{\mathop{\otimes}^L}}
\makeatother

\newcommand{\id}{\mathrm{id}}

\DeclareMathOperator{\Hom}{Hom}

\newcommand{\excise}[1]{}

\newtheorem*{conj*}{Conjecture}
\newtheorem*{thm*}{Theorem}
\newtheorem*{cor*}{Corollary}
\numberwithin{equation}{section}
\newtheorem{thm}{Theorem}[section]
\newtheorem{lem}[thm]{Lemma}
\newtheorem{prop}[thm]{Proposition}
\newtheorem{cor}[thm]{Corollary}

\theoremstyle{definition}

\theoremstyle{remark}
\newtheorem{rmk}[thm]{Remark}

\newtheorem*{rmk*}{Remark}

\DeclareMathOperator{\Rep}{Rep}

\renewcommand{\a}{\alpha}

\newcommand{\Sim}{\mathbb{L}}  
\newcommand{\hSim}{\widehat{\mathbb{L}}}
\newcommand{\Til}{\mathbb{T}}  
\newcommand{\hZ}{\widehat{Z}}
\newcommand{\hQ}{\widehat{Q}}

\newcommand{\Per}{\mathcal{P}}

\newcommand{\Wf}{W_{\mathrm{f}}}
\newcommand{\wf}{w_{\mathrm{f}}}
\newcommand{\Waff}{W}

\newcommand{\fW}{{}^{\mathrm{f}} W}
\newcommand{\fWext}{{}^{\mathrm{f}} W_{\mathrm{ext}}}
\newcommand{\Sf}{S_{\mathrm{f}}}
\newcommand{\Saff}{S}
\newcommand{\asph}{\mathrm{asph}}

\newcommand{\Haff}{\mathcal{H}}
\newcommand{\Hext}{\mathcal{H}_{\mathrm{ext}}}
\newcommand{\Hf}{\mathcal{H}_{\mathrm{f}}}
\newcommand{\Masph}{\mathcal{M}^{\asph}}
\newcommand{\Msph}{\mathcal{M}^{\mathrm{sph}}}
\newcommand{\Masphext}{\mathcal{M}^{\asph}_{\mathrm{ext}}}
\newcommand{\Msphext}{\mathcal{M}^{\mathrm{sph}}_{\mathrm{ext}}}

\newcommand{\MMasphext}{\mathsf{M}^{\asph}_{\mathrm{ext}}}

\newcommand{\MMsphext}{\mathsf{M}^{\mathrm{sph}}_{\mathrm{ext}}}
\newcommand{\puH}{{}^p \hspace{-1pt} \underline{H}}
\newcommand{\puN}{{}^p \hspace{-1pt} \underline{N}}
\newcommand{\puM}{{}^p \hspace{-1pt} \underline{M}}
\newcommand{\puP}{{}^p\hspace{-1pt} \underline{P}}

\newcommand{\pph}{{}^p \hspace{-1pt} h}

\newcommand{\uM}{\underline{M}}
\newcommand{\uH}{\underline{H}}
\newcommand{\uN}{\underline{N}}
\newcommand{\uP}{\underline{P}}
\newcommand{\sgn}{\mathsf{sgn}}
\newcommand{\triv}{\mathsf{triv}}

\def\g{\gamma}

\newcommand{\ext}{\mathrm{ext}}

\newcommand{\Afund}{A_{\mathrm{fund}}}
\newcommand{\Wext}{W_{\mathrm{ext}}}

\newcommand{\scK}{\mathscr{K}}
\newcommand{\scO}{\mathscr{O}}
\newcommand{\IW}{\mathcal{IW}}

\newcommand{\op}{\mathrm{op}}
\newcommand{\sph}{\mathrm{sph}}

\newcommand{\Sp}{\mathrm{Sp}}
\newcommand{\height}{\mathrm{ht}}

\title{A simple character formula}

\thanks{
This project has received
funding from the European Research Council (ERC) under the European Union's Horizon 2020
research and innovation programme (grant agreement No 677147).}

\author{Simon Riche}
\address{Universit\'e Clermont Auvergne, CNRS, LMBP, F-63000 Clermont-Ferrand, France.}
\email{simon.riche@uca.fr}

\author{Geordie Williamson}
\address{School of Mathematics and Statistics F07, University of
  Sydney NSW 2006, Australia.}
\email{g.williamson@sydney.edu.au}

\dedicatory{Dedicated to Jens Carsten Jantzen,\\on the occasion of his 70th birthday.}

\begin{document}

\begin{abstract}
In this paper we prove a character formula expressing the classes of simple representations in the principal block of a simply-connected semisimple algebraic group $G$ in terms of baby Verma modules, under the assumption that the characteristic of the base field is bigger than $2h-1$, where $h$ is the Coxeter number of $G$. This provides a replacement for Lusztig's conjecture, valid under a reasonable assumption on the characteristic.
\end{abstract}

\maketitle

\section{Introduction}
\label{sec:intro}

\subsection{Simple modules for reductive groups}
\label{ss:intro-simple}

Let $G$ denote a connected reductive algebraic group over an algebraically closed field
$\bk$ of characteristic $p > 0$, with simply connected derived subgroup. We fix a maximal torus and Borel
subgroup $T \subset B \subset G$. Then for every dominant weight $\lambda$ we have a
Weyl module $\Delta(\lambda)$ and its simple quotient $\Sim(\lambda)$,
both of highest weight $\lambda$. We obtain in this way a
classification of the simple algebraic $G$-modules. A central
problem in the field is to compute the characters of these simple modules.

Steinberg's tensor product theorem reduces this question to the case of
$p$-restric\-ted highest weights. For a $p$-restricted dominant weight $\lambda$ it is
known that $\Sim(\lambda)$ stays simple upon restriction to $G_1$, the
first Frobenius kernel of $G$. Moreover, all simple $G_1$-modules
occur in this way. Thus, understanding the simple $G$-modules
is equivalent to understanding the simple $G_1$-modules.

Instead of working with $G_1$-modules, it is technically more
convenient to work with $G_1T$-modules.
Simple $G_1T$-modules stay simple upon restriction to
$G_1$, and thus we can instead try to answer our question in terms of
$G_1T$-modules. Via
Brauer--Humphreys reciprocity, this question can be rephrased in terms
of indecomposable projective $G_1T$-modules (see~\S\ref{ss:intro-G1T} below for details).

In 1980 Lusztig~\cite{lusztig-patterns} proposed a conjecture for these characters if $p$
is not too small (in the
guise of ``Jantzen's generic decomposition patterns''). His conjecture
is in terms of the canonical basis in the periodic module for the
affine Hecke algebra. This formula is known to hold for large
$p$, see~\cite{kl,kt, LUSMon,ajs,fiebig}; however it is also known to fail for ``medium sized''
$p$, see~\cite{williamson}. In fact, at this point it is not known precisely when this formula holds.

Our goal in this paper is to define the $p$-canonical basis in the periodic
module and prove that the $p$-analogue of Lusztig's conjecture is true, as
long as $p \ge 2h-1$ where $h$ is the Coxeter number of $G$. Thus we obtain a character formula
for simple $G$-modules in terms of $p$-Kazhdan-Lusztig polynomials. Our proof builds on a character formula for tilting $G$-modules, proved in a joint work with P. Achar and S. Makisumi~\cite{amrw}.\footnote{One year after the first version of the present paper was made available, a different proof of this formula was obtained by the authors in~\cite{rw-smith}.}

At present, $p$-canonical bases are very difficult to compute; for
this reason our formula is certainly not the final answer to this
problem. However it gives a good conceptual understanding of where the
difficulties lie, and provides a way to compute characters or
multiplicities which is much more efficient than classical techniques
in representation theory. For instance, unpublished intensive efforts
of Jantzen were not sufficient to completely answer the question of
describing simple characters in types $\mathbf{B}_3$, 
$\mathbf{C}_3$ and $\mathbf{A}_4$. Preliminary results of
Jensen--Scheinmann indicate that our formula will allow one to
solve these cases, and maybe the case of some bigger
groups with the help of a computer. (For example, Jensen and Scheinmann
obtain a missing multiplicity in Jantzen's work
for $\mathbf{A}_4$ in a few lines using our
results.) We also believe these results will help answer the important
question of when exactly Lusztig's character formula holds.

\begin{rmk}
 \begin{enumerate}
  \item 
  A
conjecture of Donkin~\cite{donkin} would imply that our formula is valid for $p >
h$; see Remark~\ref{rmk:Donkin} below for more details.
\item
It has been known for a long time that, in theory, knowing the characters of tilting modules is enough to determine the characters of simple modules (see e.g.~\cite[\S 1.8]{rw} for details, and~\cite{sobaje} for more recent advances on this question). However, obtaining a \emph{concrete} character formula for simples out of a given character formula for tilting modules is a different story, which is the main topic of the present paper (taking as input the character formula from~\cite{amrw}).
 \end{enumerate}
\end{rmk}

\subsection{Representation theory of $G_1 T$}
\label{ss:intro-G1T}

We continue with the notation of~\S\ref{ss:intro-simple}, and
let $\bX:=X^*(T)$ be the character lattice of $T$. Let also $B^+$ be the Borel subgroup of $G$ opposite to $B$ with respect to $T$.

If as above $G_1 T$, resp.~$B_1^+ T$, denotes the preimage of the Frobenius twist $\dot{T}$ of $T$ under the Frobenius morphism of $G$, resp.~$B^+$, then
for each $\lambda \in \bX$ we have a $G_1 T$-module $\hZ(\lambda)$ (called the \emph{baby Verma module} attached to $\lambda$) obtained by coinducing to $G_1T$ the $1$-dimensional representation of $B_1^+ T$ defined by $\lambda$ (see~\S\ref{ss:G1T} below for details). The module $\hZ(\lambda)$ has a unique simple quotient $\hSim(\lambda)$, and the assignment $\lambda \mapsto \hSim(\lambda)$ induces a bijection between $\bX$ and the set of isomorphism classes of simple $G_1T$-modules. For any $\lambda,\mu \in \bX$ we have
\[
 \hSim(\lambda+p\mu) \cong \hSim(\lambda) \otimes \bk_{\dot{T}}(\mu)
\]
(where we identify the weight lattice of $\dot{T}$ with $\bX$ in such a way that the pullback under the Frobenius morphism $T \to \dot{T}$ corresponds to $\nu \mapsto p\nu$, and $\bk_{\dot{T}}(\mu)$ is viewed as a $G_1T$-module via the Frobenius morphism $G_1T \to \dot{T}$), and for $\lambda$ dominant and $p$-restricted $\hSim(\lambda)$ is the restriction to $G_1T$ of the simple $G$-module $\Sim(\lambda)$ considered in~\S\ref{ss:intro-simple}. In this way, understanding the simple $G$-modules, the simple $G_1$-modules or the simple $G_1T$ are equivalent problems.

\subsection{Characters of $G_1 T$-modules and alcoves}

Assume now that $p \geq h$.

Let $\Delta \subset \bX$ be the root system of $(G,T)$.
Let $\Wf$ and $W = \Wf \ltimes \Z \Delta$
denote the finite and affine Weyl groups, and denote their subsets of simple
reflections (determined by $B$) by $\Sf$ and $S$ respectively. The affine Weyl group acts naturally on
$\bX \otimes_\Z \R$, giving rise to the set of alcoves $\mathscr{A}$. Let
$\Afund$ denote the fundamental alcove; then the map $x \mapsto x(\Afund)$ gives a
bijection 
\begin{equation}
\label{eqn:W-alcoves}
 W \simto \mathscr{A}.
\end{equation}

We will also consider the ``dot-action'' of $W$ on $\bX$, defined by
\[
 (xt_\lambda) \cdot_p \mu = x(\mu + p\lambda + \rho)-\rho
\]
for $x \in \Wf$, $\lambda \in \Z\Delta$ and $\mu \in \bX$, where $\rho$ is the halfsum of the positive roots.
Let $\Rep_0(G_1 T)$ denote the principal block of the category of algebraic finite-dimensional $G_1T$-modules, i.e.~the Serre subcategory generated by the simple modules $\hSim(\lambda)$ with $\lambda \in W \cdot_p 0$. By construction
the highest weights of simple and baby Verma
modules belonging to this block are labelled by the affine
Weyl group, and hence (via~\eqref{eqn:W-alcoves}) by alcoves. Given an
alcove $A$, let us denote the corresponding simple and baby Verma
modules by
$\hSim_A$ and $\hZ_A$. We denote the projective cover (equivalently, injective hull) of $\hSim_A$ by
$\hQ_A$. 

Each 
$\hQ_A$ admits a baby Verma flag
(that is, a finite filtration whose successive subquotients are isomorphic to
baby Verma modues). We write $(\hQ_A:\hZ_B)$ for the number of times
the baby Verma module $\hZ_B$ occurs in such a filtration. (This number is known to be independent of the chosen filtration.) For an alcove $A \in \mathscr{A}$, consider the element
\[
q_A := \sum_{B \in \mathscr{A}}  (\hQ_A : \hZ_B ) B \in \Z[\mathscr{A}].
\]
By Brauer--Humphrey's reciprocity~\cite{humphreys}, it is known that 
\[
(\hQ_A:\hZ_B) = [\hZ_B :
\hSim_A]
\]
for any pair of alcoves $(A,B)$. As the characters of the baby Verma modules are known (and
easy!), knowledge of the elements $q_A$ for all alcoves $A$ therefore
implies knowledge of the characters of the simple $G_1T$-modules in
the principal block, and hence of all simple $G_1T$-modules by Jantzen's
translation principle.

\subsection{Statement}
\label{ss:intro-statement}

Let $\Haff$ denote the Hecke algebra of the Coxeter system $(\Waff,\Saff)$ (an algebra over the ring $\Z[v^{\pm 1}]$, with standard basis $(H_w : w \in W)$), and let
$\Per$ denote its periodic (right) module. (We follow the notational conventions of~\cite{soergel-comb-tilting}.) As a $\Z[v^{\pm
  1}]$-module, $\Per$ is free with basis given by alcoves:
\[
\Per = \bigoplus_{A \in \mathscr{A}} \Z[v^{\pm 1}] A.
\]
In~\cite{lusztig-patterns}, Lusztig has defined a canonical basis for $\Per$; following the conventions of~\cite[Theorem~4.3]{soergel-comb-tilting} we will denote this family of elements by $\{ \uP_A : A \in \mathscr{A} \}$. 
(Note that this terminology might be
  misleading: $\{ \uP_A : A \in \mathscr{A} \}$ is not a
  $\Z[v^{\pm 1}]$-basis of $\Per$, but of a certain submodule.)

Lusztig conjectured (see the last three paragraphs of the
  introduction to \cite{lusztig-patterns})
that the canonical basis determines the characters
of indecomposable projective modules, as follows:
\[
q_A = (\uP_{\hat{A}})_{ v \mapsto 1}.
\]
(Here, given $R = \sum p_A A \in \Per$, $R_{v\mapsto 1} =
\sum p_A(1) A \in \Z[\mathscr{A}]$ denotes its specialisation at $v =
1$, and $A \mapsto \hat{A}$ is a simple operation on alcoves, whose definition is recalled in Section~\ref{sec_pm}.) See also~\cite[\S 3]{fiebig-moment-graph} for the relation with Lusztig's original conjecture~\cite{lusztig} for characters of the modules $\Sim(\lambda)$.

In this paper, we define the $p$-canonical basis $\{ \puP_A : A \in \mathscr{A} \}$ in the periodic module $\Per$, and we prove that it can be used to compute the elements $q_A$, as follows.

\begin{thm}
\label{thm:main-intro}
  Assume $p \ge 2h-1$. Then for any alcove $A$ we have
\[
q_A = (\puP_{\hat{A}})_{ v \mapsto 1}.
\]
\end{thm}


\subsection{Spherical and antispherical modules}
\label{ss:intro-sph-asph}


Theorem~\ref{thm:main-intro} will be obtained as a consequence of a relation between the $p$-canonical bases in two other $\Haff$-modules, namely the (twisted) spherical and the antispherical modules. We will denote by $(\uH_w : w \in \Waff)$ the Kazhdan--Lusztig basis of $\Haff$, see~\cite[Theorem~2.1]{soergel-comb-tilting}, and by $(\puH_w : w \in W)$ its $p$-canonical basis (see~\cite{jw, rw}).

The antispherical module $\Masph$ is defined as
\[
 \Masph=\sgn \otimes_{\Hf} \Haff
\]
where $\Hf$ is the Hecke algebra of $(\Wf,\Sf)$ (which is a subalgebra in $\Haff$ in a natural way), and $\sgn$ is the right $\Hf$-module defined as $\Z[v^{\pm 1}]$ with $H_s$ acting as multiplication by $-v$ for $s \in \Sf$. This module has a standard basis $(N_w : w \in \fW)$ parametrized by the subset $\fW \subset W$ consisting of elements $w$ which are minimal in $\Wf w$, where $N_w := 1 \otimes H_w$. It also has a Kazhdan--Lusztig basis $(\uN_w : w \in \fW)$ and a $p$-canonical basis $(\puN_w : w \in \fW)$, where
\[
\uN_w = 1 \otimes \uH_w, \qquad \puN_w = 1 \otimes \puH_w.
\]
The subset $\mathscr{A}^+$ of $\mathscr{A}$ corresponding to $\fW$
under the bijection~\eqref{eqn:W-alcoves} consists of the alcoves
contained in the dominant Weyl chamber $\mathscr{C}^+$; hence we will
rather parametrize these bases by $\mathscr{A}^+$ and denote them
$(N_A : A \in \mathscr{A}^+)$, $(\uN_A : A \in \mathscr{A}^+)$ and
$(\puN_A : A \in \mathscr{A}^+)$.

On the other hand, for $\lambda \in \bX$ we denote by
$\overline{\lambda}$ the unique element in  $\overline{\Afund} \cap (W \cdot \lambda)$,
and let $W_{\overline{\lambda}} \subset W$ denote its
  isotropy group. This is a standard parabolic subgroup of $W$
  isomorphic to $\Wf$; we denote by $S_{\overline{\lambda}}:=S \cap W_{\overline{\lambda}}$ its subset of
  simple reflections, and by $\mathcal{H}_{\overline{\lambda}} \subset
  \mathcal{H}$ the associated Hecke algebra. Let $\triv
  _{\overline{\lambda}}$ denote the ``trivial'' right $\mathcal{H}_{\overline{\lambda}}$-module (defined as $\Z[v^{\pm 1}]$, with $H_s$ acting as multiplication by $v^{-1}$ for $s \in S _{\overline{\lambda}}$), and set
\[
 \Msph _{\overline{\lambda}} := \triv _{\overline{\lambda}} \otimes_{\mathcal{H}_{\overline{\lambda}}} \Haff.
\]
Then $\Msph _{\overline{\lambda}}$ has a standard basis
$(M^{\overline{\lambda}}_w := 1 \otimes H_w: w \in {}^{\overline{\lambda}}
\hspace{-1pt} \Waff)$ parametrized by the subset
${}^{\overline{\lambda}} \hspace{-1pt} \Waff \subset \Waff$ consisting
of elements $w$ which are minimal in $W _{\overline{\lambda}} w$.

If $w _{\overline{\lambda}}$ is the longest element in $W _{\overline{\lambda}}$, 
then the map $1 \otimes h \mapsto \uH_{w_{\overline{\lambda}}} \cdot
h$ provides an embedding
\begin{equation}
\label{eqn:def-zeta-lambda}
\zeta_{\overline{\lambda}}:\Msph_{\overline{\lambda}} \hookrightarrow \Haff.
\end{equation}
For any $w \in {}^{\overline{\lambda}} \hspace{-1pt} \Waff$, the
element $\uH_{w_{\overline{\lambda}} w}$,
resp.~$\puH_{w_{\overline{\lambda}} w}$, belongs to the image of
$\zeta_{\overline{\lambda}}$, and if we denote by
$\uM^{\overline{\lambda}}_w$, resp. $\puM^{\overline{\lambda}}_w$, its
preimage in $\Msph_{\overline{\lambda}}$, then
$(\uM^{\overline{\lambda}}_w : w \in {}^{\overline{\lambda}}
\hspace{-1pt} \Waff)$ and $(\puM^{\overline{\lambda}}_w : w \in
{}^{\overline{\lambda}} \hspace{-1pt} \Waff)$ are bases of
$\Msph_{\overline{\lambda}}$, called the Kazhdan--Lusztig and the
$p$-canonical basis respectively.

Recall that the set $\mathscr{A}$ also admits a natural \emph{right} action of
$\Waff$; through the identification~\eqref{eqn:W-alcoves} this
action corresponds to the right multiplication of $\Waff$ on
itself, and will be denoted $(A,x) \mapsto A \cdot x$. Then the
assignment $w \mapsto (\lambda + \Afund) \cdot w$ identifies
${}^{\overline{\lambda}} \hspace{-1pt} \Waff$ with the subset
$\mathscr{A}^+_\lambda \subset \mathscr{A}$ consisting of alcoves
contained in $\lambda+\mathscr{C}^+$. In this way the bases of
$\Msph_{\overline{\lambda}}$ considered above can be labelled by
$\mathscr{A}^+_{\lambda}$, and will be denoted $(M^\lambda_A : A \in
\mathscr{A}_{\lambda}^+)$ and $(\puM^{\lambda}_A : A \in
\mathscr{A}_{\lambda}^+)$. (Note that this labelling depends on
$\lambda$, not only on $\overline{\lambda}$.)

\begin{rmk}
\label{rmk:notation-p=0}
 The bases $(\uN_w : w \in \fW)$ and $(\uM^{\overline{\lambda}}_w : w \in
 {}^{\overline{\lambda}} \Waff)$ coincide with the bases considered in~\cite[Theorem~3.1]{soergel-comb-tilting} (for the parabolic subgroups $\Wf$ and $W_{\overline{\lambda}}$ respectively), see~\cite[Proof of Proposition~3.4]{soergel-comb-tilting}. In case $\lambda=\overline{\lambda}=0$, we will write $\Msph$ for $\Msph_0$, $M_w$ for $M_w^0$, $\uM_w$ for $\uM_w^0$, and $\puM_w$ for $\puM_w^0$.
\end{rmk}


\subsection{Outline of the proof of Theorem~\texorpdfstring{\ref{thm:main-intro}}{}}

Let $\lambda \in \bX$ be such that $\lambda + x(\Afund)$ belongs to
$\mathscr{C}^+$  for all $x \in \Wf$. It is easily seen that the assignment $1 \otimes h \mapsto \uN_{\lambda+\Afund} \cdot h$ induces an $\Haff$-module morphism
\[
 \varphi_\lambda : \Msph_{\overline{\lambda}} \to \Masph.
\]
Now we assume (for simplicity) that $G$ is semisimple, and specialize to the case $\lambda=\rho$.
The key step in our approach to Theorem~\ref{thm:main-intro} is the following claim.

\begin{thm}
\label{thm:intro-phi}
 Assume that $p$ is good for $G$. Then for any $A \in \mathscr{A}^+_\rho$ we have
 \[
  \varphi_\rho(\puM^\rho_A) = \puN_A.
 \]
\end{thm}

\begin{rmk}\phantomsection
\label{rmk:intro-Hext}
  \begin{enumerate}
   \item 
  In the body of the paper, we find it more convenient to work with
  the \emph{extended} affine Hecke algebra and its
  spherical/antispherical modules. This allows us to remove some of
  the twistings above.
   \item
   \label{it:thm-classical}
   Our proof of Theorem~\ref{thm:intro-phi} also applies to the
   Kazhdan--Lusztig bases, and shows that for any $A \in
   \mathscr{A}^+_\rho$ we have $\varphi_\rho(\uM^\rho_A) =
   \uN_A$. This fact could have been stated many years ago (since it
   does not involve the $p$-canonical bases in any way), but seems to
   be new. A direct combinatorial proof of this formula (explained to
   us by Wolfgang Soergel) is provided in Section~\ref{sec:comb-proof}.
  \end{enumerate}
\end{rmk}

Theorem~\ref{thm:intro-phi} is obtained as a ``combinatorial trace'' of a statement of categorical nature. Namely, the modules $\Msph_{\overline{\rho}}$ and $\Masph$ can be ``categorified'' via some categories of parity complexes (in the sense of~\cite{jmw}) on the affine flag variety of the Langlands dual group $G^\vee$: $\Msph_{\overline{\rho}}$ corresponds to (twisted) \emph{spherical} 
parity complexes, while $\Masph$ corresponds to
\emph{Iwahori--Whittaker}\footnote{This construction is a
  ``finite-dimensional'' and geometric counterpart for the classical
  (and very useful!) Whittaker constructions in representation theory
  of $p$-adic groups; see in particular~\cite{bbm,by,bgmrr}.} parity
complexes. The morphism $\varphi_\rho$ can also be categorified by a
functor $\Phi_\rho$, given by convolution with a certain object. We
then prove that this functor is full (but \emph{not} faithful); in
particular, as both categories involved are Krull-Schmidt, it must send indecomposable objects to indecomposable objects, and Theorem~\ref{thm:intro-phi} follows. This fullness result is deduced from a similar result in the case where the affine flag variety is replaced by the affine Grassmannian of $G^\vee$ (in which case the corresponding functor is even an equivalence of categories) proved recently by the first author with R. Bezrukavnikov, D. Gaitsgory, I. Mirkovi{\'c} and L. Rider, see~\cite{bgmrr}.

Once Theorem~\ref{thm:intro-phi} is proved, we deduce
Theorem~\ref{thm:main-intro} from the fact that the $p$-canonical
basis of $\Masph$ encodes the characters of indecomposable tilting
$G$-modules (as proved in joint work with P. Achar and
S. Makisumi~\cite{amrw}) and a result of
Jantzen~\cite{jantzen-darstellungen} (in the interpretation of
Donkin~\cite{donkin}) saying that if the highest weight of an
indecomposable tilting module is of the form $(p-1)\rho + \mu$ with
$\mu$ dominant and $p$-restricted, then this tilting module is indecomposable as a
$G_1 T$-module. It is well known (and easy to check) that these modules are projective over $G_1 T$; this therefore provides a useful relation between indecomposable tilting modules for $G$ and indecomposable projective modules for $G_1 T$, which allows us to deduce Theorem~\ref{thm:main-intro} from Theorem~\ref{thm:intro-phi} and the results of~\cite{amrw}.

\begin{rmk}
\label{rmk:Donkin}
 A conjecture by Donkin states that Jantzen's result recalled above
 should be true in any characteristic, which would imply that
 Theorem~\ref{thm:main-intro} holds as soon as $p>h$. Very recent work
 of Bendel--Nakano--Pillen--Sobaje~\cite{bnps} shows that this
 conjecture is not true in full generality. It is currently not known whether the condition that $p>h$ is sufficient to ensure that this conjecture holds (which would be sufficient for our purposes).
\end{rmk}

\subsection{Computational complexity}

Let us return briefly to our earlier claim that our formula 
is potentially useful in practice. In all known algorithms to compute
a $p$-canonical basis element $\puH_w$, $\puN_w$, $\puM_w$
etc.~the computational complexity is exponential in the length
of $w$. For a character formula involving these data to be computable
in practice, it must therefore only involve elements whose length is
small (or at least as small as possible). In the following, we compare the lengths involved in the computation based directly on the
approach of Andersen~\cite{andersen2} (see also \cite{sobaje}) and our formula.

We use $G = \Sp_4$ as a running example.\footnote{This choice is made only so that we can draw pictures. The case
  of $G = \Sp_4$ can easily be settled by classical techniques,
  e.g. the Jantzen sum formula.} Here are the first few dominant alcoves:
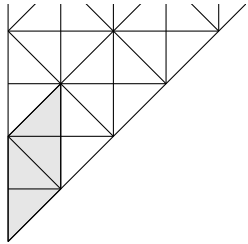
\begin{equation*}
  \begin{array}{c}
  \begin{tikzpicture}[scale=0.7]
    \begin{scope}
      \clip(-.1,-.1) rectangle (4.5,4.5);
      \draw[fill=white!90!black] (0,0) -- (0,2) -- (1,3) -- (1,1) -- (0,0);
    \draw (0,0) -- (0,8);
    \draw (1,1) -- (1,8);
    \draw (2,2) -- (2,8);
        \draw (3,3) -- (3,8);
        \draw (4,4) -- (4,8);
            \draw (5,5) -- (5,8);
        \draw (6,6) -- (6,8);
        \draw (7,7) -- (7,8);
\draw (0,0) -- (8,8);
\draw (0,2) -- (6,8);
\draw (0,4) -- (4,8);
\draw (0,6) -- (2,8);
\foreach \x in {1,2,...,8}{
  \draw (0,\x) -- (\x,\x);}
\foreach \x in {1,2,...,7}{
  \draw (\x,\x) -- (0,2*\x);}
\end{scope}
\end{tikzpicture}
 \end{array}
\end{equation*}
By Steinberg's tensor product theorem, it is enough to know the
characters of the simple modules corresponding to the shaded alcoves.

As we explained above, Brauer--Humphrey's reciprocity combined with a
result of Donkin (assuming $p \ge 2h-2$) allows us to rephrase this question in terms of the
characters of indecomposable tilting modules indexed by the following shaded alcoves:
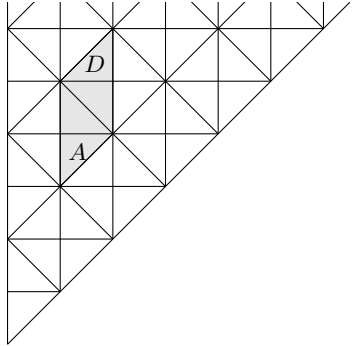
\begin{equation*}
  \begin{array}{c}
  \begin{tikzpicture}[scale=0.7]
    \begin{scope}
      \clip(-.1,-.1) rectangle (6.5,6.5);
      \draw[fill=white!90!black] (1,3) -- (1,5) -- (2,6) -- (2,4) -- (1,3);
    \draw (0,0) -- (0,8);
    \draw (1,1) -- (1,8);
    \draw (2,2) -- (2,8);
        \draw (3,3) -- (3,8);
        \draw (4,4) -- (4,8);
            \draw (5,5) -- (5,8);
        \draw (6,6) -- (6,8);
        \draw (7,7) -- (7,8);
\draw (0,0) -- (8,8);
\draw (0,2) -- (6,8);
\draw (0,4) -- (4,8);
\draw (0,6) -- (2,8);
\foreach \x in {1,2,...,8}{
  \draw (0,\x) -- (\x,\x);}
\foreach \x in {1,2,...,7}{
  \draw (\x,\x) -- (0,2*\x);}
\node at (1.33,3.66) {\small $A$}; 
\node at (1.66,5.33) {\small $D$};
\end{scope}
\end{tikzpicture}
 \end{array}
\end{equation*}
The main theorem of \cite{amrw} (see also \cite{rw-smith}) asserts that
the character of the indecomposable tilting module attached to an alcove $A \in \scA^+$
is calculated by the element $\puN_w$ for $w \in W$ such that $w(\Afund)=A$.

This looks innocent enough in this example, however for a general
group the smallest and largest alcoves in this set (marked $A$ and $D$
above) correspond to elements of $\fW$ of lengths
\begin{equation} \label{eq:orig}
2 \langle \rho^\vee, \rho \rangle = \sum_{\alpha \in \Delta^+} \height(\alpha)
\quad \text{and} \quad  4 \langle \rho^\vee, \rho \rangle - \ell(\wf) = 2
\left( \sum_{\alpha \in \Delta^+} \height(\alpha) \right) - \ell(\wf).
\end{equation}
(Here $\Delta^+ \subset \Delta$ denotes the subset of positive roots, $\height$ denotes 
height, $\rho^\vee$ denotes the halfsum of the positive
coroots and $\wf$ denotes the longest element of $\Wf$. This formula is easily deduced from standard formulas for the
length function in affine Weyl groups, see \cite[Proposition~1.23]{im}.) These numbers grow
fast in the rank of the group, and soon exceed current computational
capability.

The main technical result of this paper (Theorem \ref{thm:intro-phi}) says that the span of the $p$-canonical basis
elements indexed by the following shaded alcoves
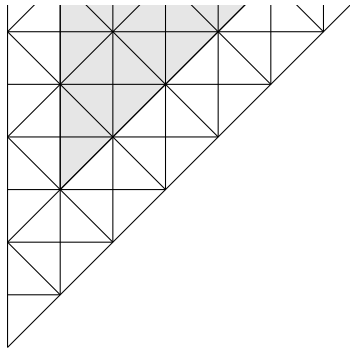
\begin{equation*}
  \begin{array}{c}
  \begin{tikzpicture}[scale=0.7]
    \begin{scope}
      \clip(-.1,-.1) rectangle (6.5,6.5);
      \draw[fill=white!90!black] (1,3) -- (1,8) -- (5,7) -- (1,3);
    \draw (0,0) -- (0,8);
    \draw (1,1) -- (1,8);
    \draw (2,2) -- (2,8);
        \draw (3,3) -- (3,8);
        \draw (4,4) -- (4,8);
            \draw (5,5) -- (5,8);
        \draw (6,6) -- (6,8);
        \draw (7,7) -- (7,8);
\draw (0,0) -- (8,8);
\draw (0,2) -- (6,8);
\draw (0,4) -- (4,8);
\draw (0,6) -- (2,8);
\foreach \x in {1,2,...,8}{
  \draw (0,\x) -- (\x,\x);}
\foreach \x in {1,2,...,7}{
  \draw (\x,\x) -- (0,2*\x);}
\end{scope}
\end{tikzpicture}
  \end{array}
\end{equation*}
in the anti-spherical module $\Masph$ yield a submodule isomorphic to
the (twisted) spherical module $\Msph_{\overline{\rho}}$, and moreover that this inclusion
preserves the $p$-canonical basis. In other words, for the tilting
characters that concern us above, it is enough to calculate the
$p$-canonical basis elements for the following shaded alcoves in the
(twisted) spherical module:
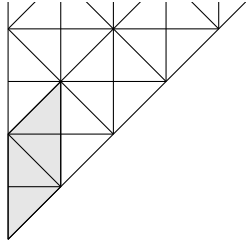
\begin{equation*}
  \begin{array}{c}
  \begin{tikzpicture}[scale=0.7]
    \begin{scope}
      \clip(-.1,-.1) rectangle (4.5,4.5);
      \draw[fill=white!90!black] (0,0) -- (0,2) -- (1,3) -- (1,1) -- (0,0);
    \draw (0,0) -- (0,8);
    \draw (1,1) -- (1,8);
    \draw (2,2) -- (2,8);
        \draw (3,3) -- (3,8);
        \draw (4,4) -- (4,8);
            \draw (5,5) -- (5,8);
        \draw (6,6) -- (6,8);
        \draw (7,7) -- (7,8);
\draw (0,0) -- (8,8);
\draw (0,2) -- (6,8);
\draw (0,4) -- (4,8);
\draw (0,6) -- (2,8);
\foreach \x in {1,2,...,8}{
  \draw (0,\x) -- (\x,\x);}
\foreach \x in {1,2,...,7}{
  \draw (\x,\x) -- (0,2*\x);}
\end{scope}
\end{tikzpicture}
 \end{array}
\end{equation*}
This is computationally a much simpler prospect, because the
elements involved are shorter; the elements we have to consider range
from length 0 to
\begin{equation} \label{eq:improve}
2 \langle \rho^\vee, \rho \rangle - \ell(\wf) = \left( \sum_{\alpha \in \Delta^+} \height(\alpha)
\right) - \ell(\wf).\end{equation}
(That is, we have improved \eqref{eq:orig} by a constant factor of $2
\langle \rho^\vee, \rho \rangle$.)

\subsection{Acknowledgements}

Part of the work on this paper was done while the first author was a
member of the Freiburg Institute for Advanced Studies, as part of the
Research Focus ``Cohomology in Algebraic Geometry and Representation
Theory'' led by A. Huber--Klawitter, S. Kebekus and W. Soergel. We
thank H.~H.~Andersen, G.~Lusztig and the referees for their
helpful suggestions. A special thanks to W.~Soergel for
providing the proof explained in Section~\ref{sec:comb-proof}.

\section{The periodic module}
\label{sec_pm}


\subsection{Extended affine Weyl group}
\label{ss:Wext}

As in~\S\ref{ss:intro-G1T} we consider a connected reductive algebraic group $G$ with simply-connected derived subgroup over an algebraically closed field $\bk$ of characteristic $p>0$, with a fixed choice of maximal torus and Borel subgroup $T \subset B \subset G$. We set $\bX:=X^*(T)$, and denote by $\Delta \subset \bX$ the root system of $(G,T)$, by $\Delta^+ \subset \Delta$ the positive system consisting of the opposites of the $T$-weights in the Lie algebra of $B$, by $\Sigma \subset \Delta$ the corresponding subset of simple roots, by $\Wf$ and $W=\Wf \ltimes \Z\Delta$ the Weyl group and the affine Weyl group, and finally by $\Sf$ and $S$ their subsets of simple roots, see e.g.~\cite[Chap.~6]{jantzen}. (Here, ``$\mathrm{f}$'' stands for ``finite.'')

By a result of Iwahori--Matsumoto~\cite{im},
the length function associated with the Coxeter system $(\Waff,\Saff)$ satisfies
\begin{equation}
\label{eqn:length} 
\ell(w \cdot t_\lambda)=\sum_{\substack{\alpha \in \Delta^+ \\ w(\alpha) \in \Delta^+}} |\langle \lambda,
\alpha^\vee \rangle | + \sum_{\substack{\alpha \in
      \Delta^+ \\ w(\alpha) \in -\Delta^+}}
  |1 + \langle \lambda, \alpha^\vee \rangle |
\end{equation}
for $w \in \Wf$ and $\lambda \in \Z\Delta$. (Here, to avoid confusion, the image of $\lambda \in \Z\Delta$ in $\Waff$ is denoted $t_\lambda$.)
We will also  consider the \emph{extended} affine Weyl group
\[
\Wext:=\Wf \ltimes \bX.
\]
The formula~\eqref{eqn:length} defines a function $\ell : \Wext \to \Z$, and we set $\Omega:=\{w \in \Wext \mid \ell(w)=0\}$. Then it is known that $\Omega$ is a subgroup of $\Wext$, and that multiplication induces a group isomorphism
\[
 \Omega \ltimes \Waff \simto \Wext.
\]
More precisely, conjugation by any element of $\Omega$ defines a Coxeter group automorphism of $\Waff$; in other words it stabilizes $\Saff$ (hence preserves lengths). It is known also that
the composition
\begin{equation}
\label{eqn:isom-Omega-fundgp}
 \Omega \hookrightarrow \Wext = W \ltimes \bX \twoheadrightarrow \bX \twoheadrightarrow \bX/\Z\Delta
\end{equation}
is a group isomorphism. (In particular, $\Omega$ is abelian.)

We consider the action of $\Waff$ and $\Wext$ on $V:=\mathbb{R} \otimes_\Z \bX$ as in~\cite[Section~4]{soergel-comb-tilting},  and denote by $\mathscr{A}$ the set of alcoves in $V$ (i.e. the connected components of the complement of the union of the reflection hyperplanes associated with the action of $\Waff$) and by $\Afund \in \mathscr{A}$ the fundamental alcove, defined as
\[
\Afund = \{v \in V \mid \forall \alpha \in \Delta^+, 0 < \langle v,\alpha^\vee \rangle < 1\}.
\]
The action of $\Wext$ on $V$ preserves alcoves, hence induces an action on $\mathscr{A}$. Moreover the assignment $w \mapsto w\Afund$ induces a bijection
$W \simto \mathscr{A}$,
see~\eqref{eqn:W-alcoves}.
If we denote by $\mathscr{A}^+ \subset \mathscr{A}$ the subset of alcoves contained in the dominant chamber (denoted $\mathcal{C}$ in~\cite{soergel-comb-tilting}), then this bijection restricts to a bijection
\begin{equation}
\label{eqn:bijection-fW-A+}
\fW \xrightarrow{\sim} \mathscr{A}^+
\end{equation}
where $\fW$ is as in~\S\ref{ss:intro-sph-asph}.

\begin{rmk}
\label{rmk:Omega}
 The subset $\Omega \subset \Wext$ can be characterized as consisting of the elements $w \in \Wext$ such that $w(\Afund)=\Afund$. For any $\lambda \in \bX$, the subset $\lambda+\Afund$ is an alcove, so that there exists $x_\lambda \in \Waff$ such that $x_\lambda(\Afund)=\lambda+\Afund$, see~\eqref{eqn:W-alcoves}. Then $\omega_\lambda=(x_\lambda)^{-1} \cdot t_\lambda$ belongs to $\Omega$, and has image $\lambda + \Z\Delta$ under~\eqref{eqn:isom-Omega-fundgp}. This procedure realizes the isomorphism inverse to~\eqref{eqn:isom-Omega-fundgp}. It also allows us to associate to any $\lambda \in \bX$ a Coxeter group automorphism $\tau_\lambda$ of $\Waff$, given by conjugation by $\omega_\lambda$ in $\Wext$. With this notation, the subset $S_{\overline{\lambda}}$ of~\S\ref{ss:intro-sph-asph} is given by $S_{\overline{\lambda}}=\tau_\lambda(\Sf)$.
\end{rmk}

\subsection{More about alcoves}

Using again the
identification~\eqref{eqn:W-alcoves} we may transport the
Bruhat order on $W$ to obtain a partial order on $\mathscr{A}$, which we also
denote by $\le$. We define the \emph{generic order} on $\mathscr{A}$
by
\[
A \preccurlyeq B \; \text{if $A + m\g \le B + m \g$ for all strictly dominant $\g$ and $m \gg 0$}.
\]
Equivalently, the generic order is uniquely determined by the fact
that it agrees with the Bruhat order on $\mathscr{A}^+$ and is
invariant under translation by $\bX$.

For any $\mu \in \bX$ we will consider the following subsets of $V$:
\begin{gather*}
\check{\Pi}_\mu := \{ v \in V \; | \; \langle \mu, \alpha^\vee \rangle
- 1 < \langle v, \alpha^\vee \rangle \le \langle \mu, \alpha^\vee
\rangle \text{ for all $\a \in \Sigma$} \}, \\
\hat{\Pi}_\mu := \{ v \in V \; | \; \langle \mu, \alpha^\vee \rangle
\le  \langle v, \alpha^\vee \rangle < \langle \mu, \alpha^\vee
\rangle + 1\text{ for all $\a \in \Sigma$} \}.
\end{gather*}
Note that for $\mu,\nu \in \bX$ we have
\begin{equation}
\label{eqn:equality-Pi}
 \check{\Pi}_\mu = \check{\Pi}_\nu \ \Leftrightarrow \ \hat{\Pi}_\mu = \hat{\Pi}_\nu \ \Leftrightarrow \ \bigl( \forall \alpha \in \Delta, \, \langle \mu-\nu, \alpha^\vee \rangle=0 \bigr).
\end{equation}
The set $\hat{\Pi}_0$ is often called the ``fundamental box.'' (It
is denoted $\Pi$ in \cite[\S 4]{soergel-comb-tilting}.) The notation is
intended to suggest that $\check{\Pi}_\mu$ (resp. $\hat{\Pi}_\mu$) is
``the fundamental box below (resp. above) $\mu$.''

Given any alcove $A \in \mathscr{A}$ there exists $\mu \in \bX$ such that
$A \subset \check{\Pi}_\mu$. The stabiliser of $\mu$ in $W$ is then
$t_\mu \Wf t_{-\mu}$. We define
\[
\hat{A} := (t_\mu \wf t_{-\mu})(A),
\]
where $\wf$ is the longest element in $\Wf$.
From~\eqref{eqn:equality-Pi} we see that $\hat{A}$ does not depend on the choice of $\mu$, that $\hat{A} \subset \hat{\Pi}_\mu$, and that $A
\mapsto \hat{A}$ is a bijection $\mathscr{A} \simto \mathscr{A}$. We
denote its inverse by $A \mapsto \check{A}$. (These operations agree
with the maps denoted similarly in~\cite{soergel-comb-tilting}, see~\cite[Comments before Lemma~4.21]{soergel-comb-tilting}.)

\subsection{The periodic module and its canonical basis}
\label{ss:periodic-module}

The \emph{periodic module} $\Per$ is the right $\Haff$-module (where, as in Section~\ref{sec:intro}, $\Haff$ is the Hecke algebra of the Coxeter system $(\Waff,\Saff)$) defined as
follows. As a $\Z[v^{\pm 1}]$-module, $\Per$ is free with basis the set of alcoves:
\[
\Per := \bigoplus_{A \in \mathscr{A}} \Z[v^{\pm 1}]A.
\]
The structure of a right $\Haff$-module on $\Per$ is characterized by the following formulas for $s \in \Saff$:
\begin{equation}
\label{eqn:action-Per}
 A \cdot \uH_s = \begin{cases}
As + vA & \text{if $A \preccurlyeq As$;} \\
As + v^{-1}A & \text{if $As \preccurlyeq A$,}
\end{cases}
\end{equation}
see~\cite[Lemma~4.1]{soergel-comb-tilting}.
(Here, $\uH_s=H_s + v$).

The action of $\bX$ on $\mathscr{A}$ by translations extends to an
action of $\bX$ on $\Per$: namely, given $R = \sum p_A A \in \Per$, set $R +
\mu := \sum p_A (\mu+A)$. This action does not commute with the $\Haff$-action; it rather satisfies the following relation for all $h \in \Haff$:
\begin{equation}
\label{eqn:trans-action}
 (R \cdot h) + \mu = (R + \mu) \cdot \tau_\mu(h),
\end{equation}
where we still denote by $\tau_\mu$ the automorphism of $\Haff$ defined by $\tau_\mu(H_w)=H_{\tau_\mu(w)}$.

As in~\S\ref{ss:intro-statement} we denote by
$\{ \uP_A : A \in \mathscr{A} \}$ the canonical basis
of $\Per$. Then for any $\mu \in \bX$ we have
\begin{gather}
  \label{eq:Ptrans}
\uP_{A + \mu} = \uP_A + \mu  
\end{gather}
(see~\cite[Comments before Proposition~4.18]{soergel-comb-tilting}) and
\begin{equation}
  \label{eq:uPAfund}
  \uP_{\Afund+\mu} = \sum_{x \in \Wf} v^{\ell(x)} \cdot \bigl( x(\Afund)+\mu \bigr)
\end{equation}
(see~\cite[Proof of Proposition~4.16]{soergel-comb-tilting}).

We now recall a crucial observation of Lusztig. Recall the embedding
$\zeta_{\overline{\mu}} : \Msph_{\overline{\mu}} \to \Haff$ from~\eqref{eqn:def-zeta-lambda}. 
By~\cite[Theorem 5.2]{lusztig-patterns}
we then have the following formula, where we set $\pi_{\mathrm{f}} := v^{-\ell(\wf)} \sum_{x \in \Wf} v^{\ell(x)}$.

\begin{lem}
\label{eq:LusztigLemma}
Let $A \in \mathscr{A}$, let $\mu \in \bX$ be such that $A \subset \hat{\Pi}_\mu$, and let $w \in \Waff$ be the unique element such that $(\mu+\Afund) \cdot w = A$. 
Then we have
\[
\uP_{A} = \frac{1}{\pi_{\mathrm{f}}} \uP_{\Afund+\mu} \cdot \zeta_{\overline{\mu}} (\uM^{\overline{\mu}}_w).
\]
\end{lem}

\begin{rmk}
\label{rmk:Lusztig-formula-translation}
The formula in Lemma~\ref{eq:LusztigLemma} is compatible with~\eqref{eq:Ptrans} in view of~\eqref{eqn:trans-action}.
\end{rmk}

\subsection{The $p$-canonical basis of the periodic module}

The usual procedure to define a $p$-canonical basis of an $\Haff$-module (see~\cite{jw, rw, ar-survey}) is to start with a categorification of this module in terms of a $\mathbb{C}$-linear category (in practice, either via some diagrammatic category or some category of parity complexes) such that the classes of indecomposable objects correspond to the Kazhdan--Lusztig basis, and then to replace (in some appropriate way) the coefficients $\mathbb{C}$ by a field of characteristic $p$. In the case of the periodic module, the known categorifications involve semi-infinite geometry, and are beyond the authors' present understanding of the subject. So we will use a different strategy to define this basis: we will start with the formula of Lemma~\ref{eq:LusztigLemma}, and replace there the canonical basis of $\Msph_{\overline{\mu}}$ by the $p$-canonical version. We expect that any reasonable categorification of $\Per$ with characteristic-$p$ coefficients should provide the same basis as the one constructed here.

Namely, for $A \in \mathscr{A}$, we choose $\mu \in \bX$ such that $A \subset \hat{\Pi}_\mu$, and let $w \in \Waff$ be the unique element such that $(\mu+\Afund) \cdot w = A$. 
Then we set
\[
\puP_{A} := \frac{1}{\pi_{\mathrm{f}}} \uP_{\Afund+\mu} \cdot \zeta_{\overline{\mu}} (\puM^{\overline{\mu}}_w).
\]
The following properties are easy to check (using in particular~\eqref{eqn:trans-action} and the fact that $\tau_\mu=\tau_\nu$ if $\hat{\Pi}_\mu=\hat{\Pi}_\nu$, as follows from~\eqref{eqn:equality-Pi}):
\begin{enumerate}
 \item
 for any $h \in \Msph_{\overline{\mu}}$ the element $\uP_{\Afund+\mu} \cdot \zeta_\mu(h)$ belongs to $\pi_{\mathrm{f}} \cdot \Per$, so that $\puP_{A}$ belongs to $\Per$;
 \item
 the element $\puP_{A}$ does not depend on the choice of $\mu$;
 \item
 for any $\nu \in \bX$ we have
 \begin{equation}
 \label{eqn:puP-trans}
 \puP_{A+\nu} = \puP_A + \nu.
 \end{equation}
\end{enumerate}
It can also be shown (although this is less obvious, and will not be proved here) that for any alcove $A$ we have
\[
 \puP_A \in \sum_{B \in \mathscr{A}} \Z_{\geq 0}[v^{\pm 1}] \cdot \uP_B.
\]

%

%

\section{The extended affine Hecke algebra and its spherical and antispherical modules}
\label{sec:Hext}

%
%

\subsection{The spherical and antispherical modules}
\label{ss:sph-asph}

We continue with the notation of Section~\ref{sec_pm}, and fix a weight $\varsigma \in \bX$ such that $\langle \varsigma, \alpha^\vee \rangle = 1$ for any $\alpha \in \Sigma$. (Such a weight exists thanks to our assumption on the derived subgroup of $G$. However, it might not be unique.)

As mentioned in Remark~\ref{rmk:intro-Hext}, to avoid difficulties related to the twists $\tau_\lambda$, it will be more convenient to work with
the Hecke algebra $\Hext$ associated with the ``quasi-Coxeter'' group $\Wext$ (see~\S\ref{ss:Wext}), i.e.~the $\Z[v^{\pm 1}]$-algebra with a ``standard'' basis consisting of elements $(H_w : w \in \Wext)$, with multiplication characterized by the following relations:
\begin{enumerate}
\item
$(H_s + v) \cdot (H_s-v^{-1})=0$ for $s \in \Saff$;
\item
$H_x \cdot H_y = H_{xy}$ if $x,y \in \Wext$ and $\ell(xy)=\ell(x)+\ell(y)$.
\end{enumerate}

The algebra $\Hext$ contains $\Haff$ as a subalgebra (spanned by the elements $H_w$ with $w \in \Waff$).
Inducing from $\Hf$ to $\Hext$ the modules considered in~\S\ref{ss:intro-sph-asph}, we obtain the right $\Hext$-modules 
\[
 \Msphext := \triv_0 \otimes_{\Hf} \Hext \quad \text{and} \quad \Masphext := \sgn \otimes_{\Hf} \Hext,
\]
which are called the \emph{spherical} and \emph{antispherical} module respectively.

We denote by $\fWext \subset \Wext$ the subset consisting of elements $w$ which are of minimal length in the coset $\Wf w$ (in other words, of the form $w \omega$ with $w \in \fW$ and $\omega \in \Omega$). Then for $w \in \fWext$ we set
\[
 M_w := 1 \otimes H_w \in \Msphext, \qquad N_w := 1 \otimes H_w \in \Masphext.
\]
The collections $(M_w : w \in \fWext)$ and $(N_w : w \in \fWext)$ are $\Z[v^{\pm 1}]$-bases of $\Msph$ and $\Masph$ respectively (called again the \emph{standard} bases). Of course there are natural embeddings $\Msph \hookrightarrow \Msph_\ext$ and $\Masph \hookrightarrow \Masph_\ext$, such that the elements denoted $M_w$, resp.~$N_w$, in~\S\ref{ss:intro-sph-asph} (see Remark~\ref{rmk:notation-p=0}) correspond to the elements denoted similarly here.

\begin{rmk}
\label{rmk:Msph-twist-ext}
 Let $\lambda \in \bX$, and consider the associated element
 $\omega_\lambda \in \Omega$, see Remark~\ref{rmk:Omega}. Then the map
 $h \mapsto H_{\omega_\lambda^{-1}} \cdot h$ induces an embedding of
 right $\Haff$-modules $\Msph_{\overline{\lambda}} \hookrightarrow \Msphext$. In
 this way, one can consider $\Msphext$ as the result of ``gluing
 together'' all the modules $\Msph_{\overline{\lambda}}$ from the introduction.
\end{rmk}

\subsection{Kazhdan--Lusztig and $p$-canonical bases}

Via the natural embeddings $\Haff \hookrightarrow \Hext$, $\Msph \hookrightarrow \Msphext$ and $\Masph \hookrightarrow \Masphext$, the Kazhdan--Lusztig and $p$-canonical bases of $\Haff$, $\Msph$ and $\Masph$ define families of elements in $\Hext$, $\Msphext$ and $\Masphext$. We complete these families into bases by setting, for $w \in \Waff$ and $\omega \in \Omega$,
\begin{gather*}
 \uH_{w \omega} := \uH_w \cdot H_\omega, \qquad \puH_{w \omega} := \puH_w \cdot H_\omega,\\
 \uM_{w \omega} := \uM_w \cdot H_\omega, \qquad \puM_{w \omega} := \puM_w \cdot H_\omega,\\
 \uN_{w \omega} := \uN_w \cdot H_\omega, \qquad \puN_{w \omega} := \puN_w \cdot H_\omega.
\end{gather*}
(It can be easily checked that we also have $\uH_{\omega w} = H_\omega \uH_w$ and $\puH_{\omega w} = H_\omega \puH_w$ for any $\omega \in \Omega$ and $w \in \Waff$.)

%

Let us recall the following well-known property of the Kazhdan--Lusztig basis.

\begin{lem}
\label{lem:uH-mult-smaller}
 Let $w \in \Wext$ and $s \in \Saff$. If $\ell(ws)<\ell(w)$, then
 \[
  \uH_w \cdot \uH_s = (v+v^{-1}) \cdot \uH_w.
 \]
\end{lem}


As in the setting of~\S\ref{ss:intro-sph-asph} we have an $\Hext$-module morphism
\[
 \xi : \Hext \to \Masphext
\]
defined by $\xi(h)= N_{\id} \cdot h$. This morphism is clearly surjective; moreover, for $w \in \Wext$ we have
\[
 \xi(\uH_w) = \begin{cases}
               \uN_w & \text{if $w \in \fWext$;} \\
               0 & \text{otherwise,}
              \end{cases} \qquad
               \xi(\puH_w) = \begin{cases}
               \puN_w & \text{if $w \in \fWext$;} \\
               0 & \text{otherwise}
              \end{cases}
\]
(see~\cite{rw} for details).

Now, let $\wf$ be the longest element in $\Wf$. Consider the endomorphism of $\Hext$ (as a right $\Hext$-module) sending $h$ to $\uH_{\wf} \cdot h$. It follows from Lemma~\ref{lem:uH-mult-smaller} that this morphism factors through a morphism
\begin{equation}
\label{eqn:zeta}
 \zeta : \Msphext \to \Hext.
\end{equation}
This morphism is injective, and satisfies
\begin{equation}
 \label{eqn:zeta-puM}
 \zeta(\uM_w)=\uH_{\wf w}, \qquad \zeta(\puM_w)=\puH_{\wf w}
\end{equation}
for any $w \in \fWext$. In particular, for $\omega \in \Omega$ we have
\begin{equation}
\label{eqn:zeta-id}
 \zeta(M_\omega) = \uH_{\wf} \cdot H_\omega = \sum_{z \in \Wf} v^{\ell(\wf)-\ell(z)} H_{z\omega}
\end{equation}
(where the second equality uses~\cite[Proposition~2.9]{soergel-comb-tilting}).

\begin{rmk}
\label{rmk:pcan-p-large}
Recall that for any $w \in \Wext$, there exists $K_w \in \Z_{\geq 0}$ such that $\puH_w = \uH_w$ for any prime number $p$ such that $p \geq K_w$. (However, determining $K_w$ is a very difficult task.) A similar claim holds for the $p$-canonical bases in $\Masphext$ and $\Msphext$.
\end{rmk}

\subsection{Statement}
\label{ss:statement}

We can now state a version of Theorem~\ref{thm:intro-phi} in terms of the \emph{extended} affine Hecke algebra (and for reductive groups which are not necessarily semisimple).

The statement will involve the element considered in the following lemma.
(Here the second equality follows from~\cite[Lemma~5.7]{soergel-comb-tilting}. We will give a (geometric) proof of both equalities in~\S\ref{ss:parity-Gr} below.)

\begin{lem}
\label{lem:rho+Af}
 We have
 \begin{equation}
 \label{eqn:uN-varsigma}
  \puN_{t_\varsigma} = \uN_{t_\varsigma} = \sum_{z \in \Wf} v^{\ell(z)} N_{t_\varsigma \cdot z}.
 \end{equation}
Moreover, for any $s \in \Sf$ we have
\[
 \uN_{t_\varsigma} \cdot \uH_s = (v+v^{-1}) \cdot \uN_{t_\varsigma}.
\]
\end{lem}

Note for later use that, if $\omega \in \Omega$, multiplying~\eqref{eqn:uN-varsigma} on the right by $H_\omega$ 
we obtain that
\begin{equation}
 \label{eqn:uN-varsigma-2}
  \puN_{t_\varsigma \omega} = \uN_{t_\varsigma \omega} = \sum_{z \in \Wf} v^{\ell(z)} N_{t_\varsigma z \omega}.
 \end{equation}

Lemma~\ref{lem:rho+Af} shows that the map $\Hext \to \Masphext$ defined by $h \mapsto \uN_{t_\varsigma} \cdot h$ factors through a morphism of right $\Hext$-modules
\begin{equation}
\label{eqn:def-phi}
  \varphi : \Msphext \to \Masphext.
\end{equation}
The main technical result of the paper is the following.

\begin{thm}
\label{thm:main}
Assume that $p$ is good for $G$. Then for any $w \in \fWext$ we have
 \[
  \varphi(\puM_w) = \puN_{t_\varsigma \cdot w}.
 \]
\end{thm}

\begin{rmk}\phantomsection
\label{rmk:main-thm}
\begin{enumerate}
\item
Theorem~\ref{thm:main} implies in particular that $\varphi$ is injective. (Of course, this can also be seen more directly.)
\item
\label{it:rmk-Haff-Hext}
To deduce Theorem~\ref{thm:intro-phi} from Theorem~\ref{thm:main}, one simply observes that $\varphi$ restricts to a morphism of right $\Haff$-modules from the submodule of $\Msphext$ generated by $N_{\omega_{\varsigma}^{-1}}$ to $\Masph$. Now the latter submodules identifies with $\Msph_{\overline{\varsigma}}$ (see Remark~\ref{rmk:Msph-twist-ext}), so that Theorem~\ref{thm:intro-phi} becomes the special case of Theorem~\ref{thm:main} when $w \in \omega_{\varsigma}^{-1} \Waff$. 
\end{enumerate}
\end{rmk}

\section{Proof of Theorem~\ref{thm:main}}

\subsection{Categorification and $p$-canonical bases}
\label{ss:categorification}

The proof of Theorem~\ref{thm:main} will use the geometric description of the $p$-canonical bases in terms of parity complexes, which we now recall. For this we need to choose a field $\K$ of coefficients for the parity complexes, which should be of characteristic $p$ but might differ from $\bk$. In fact, for technical reasons we will take for $\K$ a finite field. We also choose a prime number $\ell \neq p$, and assume that $\K$ contains a nontrivial $\ell$-th root of unity.

We now fix an algebraically closed field $\F$ of characteristic $\ell$.
Let $G^\vee$ be the connected reductive algebraic group over $\F$ which is Langlands dual to $G$. By definition, this group comes with a maximal torus $T^\vee \subset G^\vee$ whose cocharacter lattice is $\bX$. We will denote by $B^\vee \subset G^\vee$ the Borel subgroup containing $T^\vee$ whose $T^\vee$-weights are the negative coroots of $(G,T)$. We set $\scK:=\F( \hspace{-1pt} (z) \hspace{-1pt} )$, $\scO:=\F[ \hspace{-1pt} [z] \hspace{-1pt} ]$, and denote by $G^\vee_\scK$, resp.~$G^\vee_\scO$, the ind-group scheme, resp.~group scheme, over $\F$ which represents the functor $R \mapsto G^\vee \bigl( R( \hspace{-1pt} (z) \hspace{-1pt} ) \bigr)$, resp.~$R \mapsto G^\vee \bigl( R[ \hspace{-1pt} [z] \hspace{-1pt} ] \bigr)$. We also denote by $I^\vee \subset G^\vee_\scK$ the Iwahori subgroup, i.e.~the inverse image of $B^\vee$ under the evaluation morphism $G^\vee_\scO \to G^\vee$. We then consider the affine flag variety
\[
\Fl := G^\vee_\scK / I^\vee.
\]

Following~\cite{jmw} we can consider the category $\Par_{I^\vee}(\Fl,\K)$ of $I^\vee$-equivariant parity (\'etale) $\K$-complexes on $\Fl$.
The $I^\vee$-orbits on $\Fl$ are parametrized in a natural way by $\Wext$; we will denote by $\Fl_w$ the orbit corresponding to $w$ (so that $\dim(\Fl_w)=\ell(w)$). For any $w \in \Wext$, there exists a unique indecomposable parity complex $\cE_w$ on $\Fl$ which is supported on $\overline{\Fl_w}$ and whose restriction to $\Fl_w$ is $\underline{\K}_{\Fl_w}[\ell(w)]$. Then the assignment $(w,n) \mapsto \cE_w[n]$ defines a bijection between $\Wext \times \Z$ and the set of isomorphism classes of indecomposable objects in $\Par_{I^\vee}(\Fl,\K)$.

The usual convolution construction endows $\Par_{I^\vee}(\Fl,\K)$ with the structure of a monoidal category. (The fact that a convolution of parity complexes is parity is proved in~\cite[\S 4.1]{jmw}.) In particular, the split Grothendieck group
\[
[\Par_{I^\vee}(\Fl,\K)]
\]
has a natural product; we will in fact view this ring as a $\Z[v^{\pm 1}]$-algebra, where $v$ acts via the automorphism induced by the cohomological shift $[1]$. It is well known (see~\cite{springer,jmw,jw}) that there exists a unique $\Z[v^{\pm 1}]$-algebra isomorphism
\begin{equation}
\label{eqn:Hext-parity}
\Hext \simto [\Par_{I^\vee}(\Fl,\K)]
\end{equation}
sending $\uH_s$ to $[\cE_s]$ for any $s \in \Saff$ and $H_\omega$ to $[\cE_\omega]$ for any $\omega \in \Omega$.
Then
for $w \in \Wext$, the element $\puH_w$ is the inverse image of $[\cE_w]$ under~\eqref{eqn:Hext-parity}, see e.g.~\cite[Part~3]{rw}.
Recall also that the \emph{$p$-Kazhdan--Lusztig polynomials} are the elements $(\pph_{y,w})_{y,w \in \Wext}$ of $\Z[v^{\pm 1}]$ such that
\[
 \puH_w = \sum_{y \in \Wext} \pph_{y,w} \cdot H_y.
\]


In order to categorify the module $\Masphext$, we consider the category of ``Iwahori--Whittaker'' parity complexes $\Par_{\IW}(\Fl,\K)$ on $\Fl$. These objects are defined using the action of the unipotent radical $I^{\vee,+}_{\mathrm{u}}$ of the Iwahori subgroup $I^{\vee,+}$ associated with the Borel subgroup $B^{\vee,+}$ of $G^\vee$ which is opposite to $B^\vee$ with respect to $T^\vee$; see~\cite[\S 11]{rw} for details. (Here we use our assumption on $\ell$-th roots of unity in $\K$.) The $I^{\vee,+}_{\mathrm{u}}$-orbits on $\Fl$ are parametrized in a natural way by $\Wext$; but only those corresponding to elements in $\fWext$ support nonzero Iwahori--Whittaker local systems. Therefore the isomorphism classes of indecomposable objects in $\Par_{\IW}(\Fl,\K)$ are naturally in bijection with $\fWext \times \Z$; we will denote by $\cE_w^\IW$ the object associated with $(w,0)$.

The convolution construction defines a right action of the monoidal category $\Par_{I^\vee}(\Fl,\K)$ on the category $\Par_{\IW}(\Fl,\K)$, and there exists a unique isomorphism of right $\Hext$-modules
\begin{equation}
\label{eqn:Masphext}
\Masphext \simto [\Par_{\IW}(\Fl,\K)]
\end{equation}
sending $N_{\id}$ to $\cE^\IW_{\id}$. Then for $w \in \fWext$, $\puN_w$ is the inverse image of $[\cE_w^\IW]$ under this isomorphism, see~\cite[\S 11]{rw}.



Finally, we explain the categorification of $\Msphext$.
We consider the ``opposite affine Grassmannian''
\[
\Gr^\op := G^\vee_\scO \backslash G^\vee_{\scK}.
\]
This variety admits an action of $I^\vee$ induced by right multiplication on $G^\vee_\scK$, and we can consider the corresponding category of parity complexes $\Par_{I^\vee}(\Gr^{\op},\K)$. The $I^\vee$-orbits on $\Gr^{\op}$ are parametrized by $\fWext$; therefore the indecomposable objects in $\Par_{I^\vee}(\Gr^{\op},\K)$ are parametrized in a natural way by $\fWext \times \Z$. The object associated with $(w,0)$ (for $w \in \fWext$) will be denoted $\cF_w$.

Again, the convolution construction defines a right action of the monoidal category $\Par_{I^\vee}(\Fl,\K)$ on the category $\Par_{I^\vee}(\Gr^{\op},\K)$, and there exists a unique isomorphism of right $\Hext$-modules
\[
\Msphext \simto [\Par_{I^\vee}(\Gr^{\op},\K)]
\]
sending $M_{\id}$ to $\cF_{\id}$. Using~\cite[Lemma~A.5]{acr} and the construction of the $p$-canonical basis in $\Msphext$, one can check that,
for $w \in \fWext$, $\puM_w$ is the inverse image of $[\cF_w]$ under this isomorphism.


\subsection{Parity complexes on affine Grassmannians}
\label{ss:parity-Gr}

From now on we assume that $p$ is good for $G$.

Consider the ``usual'' affine Grassmannian
\[
\Gr:=G^\vee_{\scK} / G^\vee_\scO,
\]
and the $G^\vee_\scO$-equivariant constructible derived category $\Db_{G^\vee_\scO}(\Gr,\K)$. This category possesses a natural perverse t-structure, whose heart will be denoted $\Perv_{G^\vee_\scO}(\Gr, \K)$.

Under our assumptions that $p$ is good for $G$ (equivalently, for $G^\vee$) and that $G$ has a simply-connected derived subgroup (so that the quotient of $X^*(T^\vee)$ by the coroot lattice of $G$ is torsion-free), it is known that the equivariant cohomology $\mathsf{H}^\bullet_{G^\vee}(\mathrm{pt}; \K) = \mathsf{H}^\bullet_{G^\vee_\scO}(\mathrm{pt}; \K)$ vanishes in odd degrees; see e.g.~\cite[\S 2.6]{jmw} or~\cite[\S 3.2]{mr} for references. Therefore, the theory developed in~\cite{jmw} applies in this context, and we will denote by $\Par_{G^\vee_\scO}(\Gr,\K)$ the corresponding category of parity complexes.

For $\lambda \in \bX$, we set $L_\lambda := z^{\lambda} \cdot G^\vee_\scO \in \Gr$. Then the assignment $\lambda \mapsto G^\vee_\scO \cdot L_\lambda$ induces a bijection between $\bX^+$ and the set of $G^\vee_\scO$-orbits on $\Gr$. Therefore, the isomorphism classes of indecomposable objects in $\Par_{G^\vee_\scO}(\Gr,\K)$ are parametrized in a natural way by $\bX^+ \times \Z$; for $\lambda \in \bX^+$ we will denote by $\cE^\sph_\lambda$ the object associated with $(\lambda,0)$.

The following result is proved in~\cite{jmw2} under some technical assumptions, and in~\cite[Corollary~1.6]{mr} in the present generality. (This claim is known to be false if we remove the assumption that $p$ is good for $G$, see~\cite{jmw2}.)

\begin{thm}
\label{thm:MR}
For any $\lambda \in \bX^+$, the object $\cE^\sph_\lambda$ is perverse.
\end{thm}

\begin{rmk}
 From the combinatorial point of view, this theorem says that if $w \in \Wext$ is maximal in $\Wf w \Wf$, then $\puH_w$ belongs to $\bigoplus_y \Z \cdot \uH_y$.
\end{rmk}

The other result which we will need is the main result
of~\cite{bgmrr}. Here we consider the Iwahori--Whittaker derived
category of sheaves on $\Gr$, denoted $\Db_{\IW}(\Gr,\K)$. This
category is endowed with the perverse t-structure, whose heart will be
denoted $\Perv_{\IW}(\Gr,\K)$. This abelian category admits a natural
structure of highest weight category (in the sense considered e.g.~in~\cite[\S 7]{riche-hab}), and moreover the realization functor provides an equivalence of triangulated categories
\[
 \Db \Perv_{\IW}(\Gr,\K) \simto \Db_{\IW}(\Gr,\K).
\]

 The $I^{\vee,+}_{\mathrm{u}}$-orbits on $\Gr$ are parametrized by $\bX$, and those which support a nonzero Iwahori--Whittaker local system are the ones parametrized by elements in $\varsigma + \bX^+$ (i.e.~by strictly dominant weights). In particular, no orbit in the boundary of the orbit associated with $\varsigma$ supports such a local system; therefore the corresponding standard perverse sheaf is simple, and isomorphic to the associated costandard perverse sheaf (see~\cite[Equation~(3.2)]{bgmrr}). Hence this object is also parity, and will be denoted $\cF_\varsigma^\IW$.
 
 The following result is proved in~\cite{bgmrr}. (The first claim holds without any assumption on $p$; however for the second assertion we need the restriction that $p$ is good.) Here we denote by $\star^{G^\vee_\scO}$ the natural convolution bifunctor
 \[
  \Db_\IW(\Gr,\K) \times \Db_{G^\vee_\scO}(\Gr, \K) \to \Db_\IW(\Gr,\K)
 \]
(see~\cite{bgmrr} for details).
 
 \begin{thm}
 \label{thm:BGMRR}
 The functor
 \[
 \Psi : \Db_{G^\vee_\scO}(\Gr, \K) \to \Db_\IW(\Gr,\K)
 \]
 defined by $\Psi(\cF) = \cF_\varsigma^\IW \star^{G^\vee_\scO} \cF$ is t-exact for the perverse t-structures, and restricts to an equivalence of categories $\Perv_{G^\vee_\scO}(\Gr, \K) \simto \Perv_\IW(\Gr,\K)$.
Moreover, for any $\lambda \in \bX$, the object $\Psi(\cE^\sph_\lambda)$ is a tilting perverse sheaf.
 \end{thm}
 
 The consequence of Theorems~\ref{thm:MR} and~\ref{thm:BGMRR} that we will use below is the following.
 
\begin{cor}
\label{cor:Phi-full}
For any $\cG,\cG'$ in $\Par_{G^\vee_\scO}(\Gr, \K)$, the morphism
\[
\Hom_{\Db_{G^\vee_\scO}(\Gr, \K)}(\cG,\cG') \to \Hom_{\Db_{\IW}(\Gr,\K)}(\Psi(\cG), \Psi(\cG'))
\]
induced by $\Psi$
is surjective.
\end{cor}

\begin{proof}
Any object of $\Par_{G^\vee_\scO}(\Gr, \K)$ is a direct sum of cohomological shifts of objects of the form $\cE^\sph_\lambda$ (with $\lambda \in \bX^+$); therefore to prove the corollary it suffices to prove that for any $\lambda,\mu \in \bX^+$ and $n \in \Z$ the functor $\Psi$ induces a surjection
\begin{equation}
\label{eqn:Phi-full}
\Hom_{\Db_{G^\vee_\scO}(\Gr, \K)}(\cE^\sph_\lambda,\cE^\sph_\mu[n]) \to \Hom_{\Db_{\IW}(\Gr,\K)}(\Psi(\cE^\sph_\lambda), \Psi(\cE^\sph_\mu)[n]).
\end{equation}
Now, by Theorem~\ref{thm:BGMRR} the objects $\Psi(\cE^\sph_\lambda)$ and $\Psi(\cE^\sph_\mu)$ are tilting perverse sheaves; therefore the right-hand side vanishes unless $n=0$. And if $n=0$, since $\cE^\sph_\lambda$ and $\cE^\sph_\mu$ are perverse, and since $\Psi$ restricts to an equivalence on perverse sheaves, the map~\eqref{eqn:Phi-full} is an isomorphism in this case.
\end{proof}

We can now give the proof of Lemma~\ref{lem:rho+Af}.

\begin{proof}[Proof of Lemma~{\rm \ref{lem:rho+Af}}]
One can easily check using~\eqref{eqn:length} (and the fact that $\ell(x)=\ell(x^{-1})$ for any $x \in \Wext$) that $t_\varsigma$ is of maximal length in $t_\varsigma \cdot \Wf$. Therefore, the $I^{\vee,+}_{\mathrm{u}}$-orbit in $\Fl$ associated with $t_\varsigma$ is the inverse image under the projection $\pi : \Fl \to \Gr$ of the orbit of $L_\varsigma$. Using~\cite[Lemma~A.5]{acr} we deduce that $\cE_{t_\varsigma}^\IW = \pi^*(\cF_\varsigma^\IW) [\ell(\wf)]$. It follows that $  \puN_{t_\varsigma} = \sum_{z \in \Wf} v^{\ell(z)} N_{t_\varsigma \cdot z}$, and that for any $s \in \Sf$ we have
 $\puN_{t_\varsigma} \cdot \uH_s = (v+v^{-1}) \cdot \puN_{t_\varsigma}$. The claims about $\uN_{t_\varsigma}$ follow, taking $p \gg 0$ (see Remark~\ref{rmk:pcan-p-large}).
\end{proof}

\subsection{Fullness}

To prove Theorem~\ref{thm:main} we will consider a categorification of $\varphi$. 
For this, we work with the $G^\vee_\scO$-equivariant derived category $\Db_{G^\vee_\scO}(\Fl,\K)$. This category admits a right action of the $I^\vee$-equivariant derived category $\Db_{I^\vee}(\Fl,\K)$ (by convolution, as usual), and it is clear that there exists a canonical equivalence of triangulated categories
\[
\imath : \Db_{I^\vee}(\Gr^\op,\K) \simto \Db_{G^\vee_\scO}(\Fl,\K)
\]
sending $\cF_\id$ to $\underline{\K}_{G^\vee_\scO/I^\vee}[\ell(\wf)]$ (where $\Gr^\op$ is as in~\S\ref{ss:categorification}) and commuting with the right actions of $\Db_{I^\vee}(\Fl,\K)$ on both sides.
Moreover, the theory of parity complexes from~\cite{jmw} applies in $\Db_{G^\vee_\scO}(\Fl,\K)$ also, and $\imath$ restricts to an equivalence of categories
\[
\imath_{\mathsf{Par}} : \Par_{I^\vee}(\Gr^\op,\K) \simto \Par_{G^\vee_\scO}(\Fl,\K),
\]
where in the right-hand side $\Par_{G^\vee_\scO}(\Fl,\K)$ means the full subcategory of parity complexes 
in $\Db_{G^\vee_\scO}(\Fl,\K)$.
In particular, the indecomposable objects in the category $\Par_{G^\vee_\scO}(\Fl,\K)$ are the objects $\imath(\cF_w)[n]$ for $w \in \fWext$ and $n \in \Z$, and we have a canonical isomorphism
\[
\Msphext \simto [\Par_{G^\vee_\scO}(\Fl,\K)]
\]
sending $\puM_w$ to $[\imath(\cF_w)]$ for any $w \in \fWext$.

We now consider the functor
\[
\Phi : \Db_{G^\vee_{\scO}}(\Fl,\K) \to \Db_\IW(\Fl,\K)
\]
defined by
\[
\Phi(\cF) = \cF^\IW_{\varsigma} \star^{G^\vee_\scO} \cF,
\]
where $\star^{G^\vee_\scO}$ now denotes the natural convolution bifunctor
\[
\Db_{\IW}(\Gr,\K) \times \Db_{G^\vee_\scO}(\Fl,\K) \to \Db_{\IW}(\Fl,\K).
\]

\begin{lem}
\label{lem:Psi-parity}
The functor $\Phi$ sends parity complexes to parity complexes. Moreover, the map on split Grothendieck groups induced by the restriction
\[
\Phi_{\mathsf{Par}} : \Par_{G^\vee_{\scO}}(\Fl,\K) \to \Par_\IW(\Fl,\K)
\]
is $\varphi$.
\end{lem}

\begin{proof}
The proof of the first claim is similar to that of~\cite[Lemma~4.14]{bgmrr}. For the second claim, we observe that the map induced by $\Phi_{\mathsf{Par}}$ is clearly a morphism of right $\Hext$-modules. Since $\Msphext$ is a cyclic module, this reduces the proof to checking that the image of $[\underline{\K}_{G^\vee_\scO/I^\vee}[\ell(\wf)]]$ corresponds to $\varphi(M_e)=\uN_{t_\varsigma}$ under~\eqref{eqn:Masphext}. However we have
\[
\Phi \bigl( \underline{\K}_{G^\vee_\scO/I^\vee}[\ell(\wf)] \bigr) = \pi^*(\cF^\IW_{\varsigma}) [\ell(\wf)],
\]
where $\pi : \Fl \to \Gr$ is the projection. As observed in the proof of Lemma~\ref{lem:rho+Af}, the right-hand side is $\cE^{\IW}_{t_\varsigma}$, whose class in the Grothendieck group corresponds to $\puN_{t_\varsigma}$ by definition. The claim follows, using the first equality in Lemma~\ref{lem:rho+Af}.
\end{proof}

The key step in our proof of Theorem~\ref{thm:main} is the following claim.

\begin{prop}
The functor $\Phi_{\mathsf{Par}}$ from Lemma~{\rm \ref{lem:Psi-parity}}
is full.
\end{prop}

\begin{proof}
It is easily seen that any object in $ \Par_{G^\vee_{\scO}}(\Fl,\K)$ is a direct sum of direct summands of objects of the form $\underline{\K}_{G^\vee_\scO/I^\vee} \star^{I^\vee} \cE$ with $\cE$ in $\Par_{I^\vee}(\Fl,\K)$, where $\star^{I^\vee}$ is the natural convolution bifunctor
\[
 \Par_{G^\vee_{\scO}}(\Fl,\K) \times \Par_{I^\vee}(\Fl,\K) \to \Par_{G^\vee_{\scO}}(\Fl,\K).
\]
Now any functor of the form $(-) \star^{I^\vee} \cE$ (with $\cE$ in $\Par_{I^\vee}(\Fl,\K)$) admits a left adjoint of the form $(-) \star^{I^\vee} \cE'$ with $\cE'$ in $\Par_{I^\vee}(\Fl,\K)$, hence this remark reduces the proof of fullness of $\Phi_{\mathsf{Par}}$ to proving that for any $\cF$ in $\Par_{G^\vee_{\scO}}(\Fl,\K)$ the map
\[
\Hom^\bullet_{\Db_{G^\vee_{\scO}}(\Fl,\K)}(\cG, \underline{\K}_{G^\vee_\scO/I^\vee}) \to \Hom^\bullet_{\Db_{\IW}(\Fl,\K)}(\Phi(\cG), \Phi(\underline{\K}_{G^\vee_\scO/I^\vee}))
\]
induced by $\Phi$ is surjective. 
If $\pi$ is as in the proof of Lemma~\ref{lem:rho+Af} (or of Lemma~\ref{lem:Psi-parity}), then we have
\[
\underline{\K}_{G^\vee_\scO/I^\vee} = \pi^* \cE^{\sph}_0, \quad \Phi(\underline{\K}_{G^\vee_\scO/I^\vee}) = \pi^*(\cF_\varsigma^\IW).
\]
Since $\pi^* \cong \pi^![-2\ell(\wf)]$, using adjunction we deduce isomorphisms
\begin{align*}
\Hom^\bullet_{\Db_{G^\vee_{\scO}}(\Fl,\K)}(\cG, \underline{\K}_{G^\vee_\scO/I^\vee}) &\cong \Hom^{\bullet-2\ell(\wf)}_{\Db_{G^\vee_{\scO}}(\Gr,\K)}(\pi_!(\cG), \cE^\sph_0), \\
 \Hom^\bullet_{\Db_{\IW}(\Fl,\K)}(\Phi(\cG), \Phi(\underline{\K}_{G^\vee_\scO/I^\vee})) &\cong  \Hom^{\bullet-2\ell(\wf)}_{\Db_{\IW}(\Gr,\K)}(\pi_! \Phi(\cG), \cF_\varsigma^\IW).
\end{align*}
Now we have $\pi_! \Phi(\cG) \cong \Psi(\pi_! \cG)$, where $\Psi$ is as in Theorem~\ref{thm:BGMRR}; hence we are reduced to proving that the morphism
\[
\Hom^{\bullet}_{\Db_{G^\vee_{\scO}}(\Gr,\K)}(\pi_!(\cG), \cE^\sph_0) \to \Hom^{\bullet}_{\Db_{\IW}(\Gr,\K)}(\Psi(\pi_! \cG), \cF_\varsigma^\IW)
\]
induced by $\Psi$ is surjective. However, $\pi_! \cG$ is parity, so that the claim follows from Corollary~\ref{cor:Phi-full}.
\end{proof}

\subsection{Proof of Theorem~\ref{thm:main}}

Since the functor $\Phi$ is a full functor between Krull-Schmidt categories, it must send indecomposable objects to indecomposable
objects. Indeed, this follows from the observation that any quotient
of a local ring is local. Using 
support considerations
it is not difficult to deduce that for any $w \in \fWext$ we have
\[
\Phi(\imath(\cF_w)) \cong \cE^\IW_{t_\varsigma w}.
\]
Passing to classes in the split Grothendieck group we deduce the formula of Theorem~\ref{thm:main}.

\section{Application: a character formula for simple \texorpdfstring{$G$}{G}-modules}


In this section we return to the setting of Sections~\ref{sec:intro}--\ref{sec:Hext}; in particular, $G$ is a connected reductive algebraic group with simply connected derived subgroup over an algebraically closed field $\bk$ of characteristic $p$. We will assume that $p>h$, where $h$ is the Coxeter number of $G$. (In particular, this condition implies that $p$ is good for $G$, so that Theorem~\ref{thm:main} is applicable.)

\subsection{The tilting character formula}

We set
\[
 \MMasphext := \Z \otimes_{\Z[v^{\pm 1}]} \Masphext,
\]
where $\Z$ is considered as a $\Z[v^{\pm 1}]$-module via $v \mapsto 1$. This $\Z$-module is a right module over
\[
 \Z \otimes_{\Z[v^{\pm 1}]} \Hext = \Z[\Wext].
\]

We will denote by $\Rep(G)$ the abelian category of finite-dimensional algebraic $G$-modules. The simple objects in this category are labelled in a natural way by the subset $\bX^+ \subset \bX$ of dominant weights; as in~\S\ref{ss:intro-simple} we will denote by $\Sim(\lambda)$ the simple $G$-module of highest weight $\lambda \in \bX^+$.

We consider the dilated and shifted action of $\Waff$ on $\bX$ defined by
\[
 w \cdot_p \lambda = w(\lambda+\varsigma)-\varsigma, \quad t_\mu \cdot_p \lambda = \lambda + p\mu
\]
for $w \in \Wf$ and $\lambda,\mu \in \bX$. (It is a classical fact that this action does not depend on the choice of $\varsigma$.) We then denote by $\Rep_\varnothing(G)$ the ``extended principal block'' of $\Rep(G)$, i.e.~the Serre subcategory generated by the simple objects $\Sim(w \cdot_p 0)$ with $w \in \fWext$. (Here, under our assumptions, for $w \in \Wext$ we have $w \cdot_p 0 \in \bX^+$ iff $w \in \fWext$.) 

For $\lambda \in \bX^+$ we also denote by $\Delta(\lambda)$, $\nabla(\lambda)$ and $\Til(\lambda)$ the Weyl, induced, and indecomposable tilting $G$-modules of highest weight $\lambda$ (see~\cite[\S 3.1]{rw}). If $\lambda = w \cdot_p 0$ for some $w \in \fWext$, then these objects belong to $\Rep_\varnothing(G)$.

In the following lemma, $T_\mu^\nu$ is the translation functor from the $\mu$-block to the $\nu$-block of $\Rep(G)$, see~\cite[Chapter~II.7]{jantzen}. 
See also Remark~\ref{rmk:Omega} for the definition of $\omega_\varsigma$.

\begin{lem}
\label{lem:Til-rho+Afund}
 For any $\omega \in \Omega$ we have
 \[
\Til(p\varsigma + \omega \cdot_p 0) \cong T_{\omega_\varsigma \cdot_p(-\varsigma)}^{\omega_\varsigma \omega \cdot_p 0} \bigl( \Sim ((p-1)\varsigma) \bigr).
 \]
 Moreover, for any $\lambda \in \bX^+$
we have
 \[
  (\Til(p\varsigma + \omega \cdot_p 0) : \nabla(\lambda)) = \begin{cases}
                                       1 & \text{if $\lambda=t_\varsigma x \omega \cdot_p 0$ for some $x \in \Wf$;} \\
                                       0 & \text{otherwise.}
                                      \end{cases}
 \]

\end{lem}

\begin{proof}
We have $p\varsigma + \omega \cdot_p 0 = t_\varsigma \omega_{\varsigma}^{-1} \cdot_p (\omega_\varsigma \omega \cdot_p 0)$, and $(p-1)\varsigma = t_\varsigma \omega_{\varsigma}^{-1} \cdot_p (\omega_\varsigma \cdot_p(-\varsigma))$. Moreover, $p\varsigma + \omega \cdot_p 0$ is maximal among the elements of the form $t_\varsigma \omega_{\varsigma}^{-1} \cdot_p (\omega_\varsigma x \omega \cdot_p 0)$ with $x \in \Wf$. Hence by~\cite[Proposition~E.11]{jantzen} we have
 \[
\Til(p\varsigma + \omega \cdot_p 0) \cong T_{\omega_\varsigma \cdot_p(-\varsigma)}^{\omega_\varsigma \omega \cdot_p 0} \bigl( \Til ((p-1)\varsigma) \bigr).
 \] 
Now by~\cite[Remark in~\S II.3.19]{jantzen}, the Steinberg module $\Sim((p-1)\varsigma)$ is isomorphic to $\Delta((p-1)\varsigma)$ and to $\nabla((p-1)\varsigma)$, hence is tilting. It is also clearly indecomposable, so that $\Til((p-1)\varsigma)=\Sim((p-1)\varsigma)$. The first claim follows. 

The second claim follows from the first one (and the fact that $\Sim((p-1)\varsigma)=\nabla((p-1)\varsigma)$) in view of~\cite[Proposition~II.7.13]{jantzen}.
\end{proof}

For any $s \in S$ we choose a weight $\mu_s$ as in~\cite[\S 3.1]{rw} (i.e. a ``generic'' weight on the $s$-wall of the fundamental alcove for the dilated and shifted action) and consider the exact selfadjoint endofunctor
\[
 \Theta_s := \bigoplus_{\omega \in \Omega} T_{\omega \cdot_p \mu_s}^{\omega \cdot_p 0} T_{\omega \cdot_p 0}^{\omega \cdot_p \mu_s}
\]
of $\Rep_\varnothing(G)$.
(Here the sum might be infinite but, for each object $V$ of $\Rep_\varnothing(G)$, only finitely many of these functors applied to $V$ do not vanish; so the functor $\Theta_s$ is well defined.)
If we denote by $[\Rep_\varnothing(G)]$ the Grothendieck group of the abelian category $\Rep_\varnothing(G)$, and by $[M]$ the class of an object $M$, then it 
is well known (see e.g.~\cite[\S 1.2]{rw}) 
that the assignment $1 \otimes N_w \mapsto [\Delta(w \cdot_p 0)]=[\nabla(w \cdot_p 0)]$ induces an isomorphism of abelian groups
\begin{equation}
\label{eqn:Groth-group}
 \MMasphext \simto [\Rep_\varnothing(G)].
\end{equation}
Using~\cite[Propositions~II.7.11 and~II.7.12]{jantzen} one can check that,
under this identification, the endomorphism of the right-hand side induced by $\Theta_s$ corresponds to the action of $(1 + s)$ on the left-hand side.

The following statement was conjectured in~\cite{rw} (see in particular~\cite[Corollary~1.4.1]{rw}) and proved in~\cite{amrw}.

\begin{thm}
\label{thm:amrw}
 Under the isomorphism~\eqref{eqn:Groth-group}, $1 \otimes \puN_w$ is sent to $[\Til(w \cdot_p 0)]$ for any $w \in \fWext$.
\end{thm}

\begin{rmk}
In~\cite{rw,amrw} we work with $\Waff$ instead of $\Wext$; but the extension is immediate (see e.g.~\cite{ar-survey}).
\end{rmk}

\subsection{$G_1 T$-modules}
\label{ss:G1T}

Let $\dot{G}$ be the Frobenius twist of $G$, and let $\mathrm{Fr} : G \to \dot{G}$ be the Frobenius morphism. We denote by $G_1$ the kernel of $\mathrm{Fr}$ (a normal finite subgroup scheme of $G$) and by $G_1 T$, resp.~$G_1 B^+$, the inverse image of the Frobenius twist $\dot{T}$ of $T$, resp.~$\dot{B}^+$ of $B^+$, under $\mathrm{Fr}$. (Here, $B^+$ is the Borel subgroup of $G$ opposite to $B$ with respect to $T$.) We will identify the characters of $\dot{T}$ with $\bX$, in such a way that the composition of the Frobenius morphism $T \to \dot{T}$ with the character $\lambda$ of $\dot{T}$ is the character $p \lambda$ of $T$. We use similar conventions for $\dot{B}^+$.

We will denote by $\Rep(G_1 T)$ the category of finite-dimensional algebraic $G_1 T$-modules. As explained in~\cite[Proposition~II.9.6]{jantzen}, the simple objects in $\Rep(G_1 T)$ are in a canonical bijection with $\bX$; the simple module corresponding to $\lambda$ will be denoted $\hSim(\lambda)$. If $\lambda$ is dominant and restricted then $\hSim(\lambda)$ is the restriction of $\Sim(\lambda)$ to $G_1 T$, and if $\lambda,\mu \in \bX$ we have
\begin{equation}
\label{eqn:translation-simples}
\hSim(\lambda+p\mu) \cong \hSim(\lambda) \otimes \bk_{\dot T}(\mu),
\end{equation}
where $\bk_{\dot T}(\mu)$ is seen as a $G_1 T$-module via the surjection $G_1 T \to \dot T$. In particular, these simple objects are completely determined by the objects $(\Sim(\lambda) : \lambda \in \bX^+)$.

The category $\Rep(G_1 T)$ also contains the \emph{baby Verma modules}
\[
 \hZ(\lambda) = \mathrm{Coind}_{B^+}^{G_1 B^+} \bigl( \bk_{B^+} (\lambda) \bigr)
\]
for $\lambda \in \bX$.
With these conventions, we have a surjection of $G_1 T$-modules $\hZ(\lambda) \twoheadrightarrow \hSim(\lambda)$, see~\cite[Proposition~9.6(d)]{jantzen}. We also have canonical isomorphisms
\begin{equation}
\label{eqn:translation-babyVerma}
  \hZ(\lambda + p\mu) \cong \hZ(\lambda) \otimes \bk_{\dot{T}}(\mu)
\end{equation}
for $\lambda, \mu \in \bX$.

The category $\Rep(G_1 T)$
admits a ``block'' decomposition similar to that of $\Rep(G)$; see e.g.~\cite[\S II.9.19]{jantzen}. Hence we can consider $\Rep_\varnothing(G_1 T)$, the ``extended block'' of the weight $0$, i.e.~the Serre subcategory generated by the simple objets $\hSim(\lambda)$ with $\lambda \in \Wext \cdot_p 0$. If $\lambda \in \Wext \cdot_p 0$ then 
$\hZ(\lambda)$ also belongs to $\Rep_\varnothing(G_1 T)$.
It is clear from the Steinberg tensor product formula~\cite[Proposition~II.3.16]{jantzen} that the restriction functor $\Rep(G) \to \Rep(G_1 T)$ restricts to a functor
\begin{equation}
\label{eqn:Res-0}
 \Rep_\varnothing(G) \to \Rep_\varnothing(G_1 T).
\end{equation}
The image of a $G$-module $M$ under this functor will sometimes be denoted $M_{|G_1 T}$.

As explained in~\cite[\S II.9.22]{jantzen}, the translation functors can be canonically ``lifted'' to the category $\Rep(G_1 T)$. In particular, this means that there exists an exact selfadjoint endofunctor of $\Rep_\varnothing(G_1 T)$, which for simplicity will also be denoted $\Theta_s$, and such that the following diagram commutes:
\[
 \xymatrix@C=2cm{
 \Rep_\varnothing(G) \ar[d]_-{\eqref{eqn:Res-0}} \ar[r]^-{\Theta_s} & \Rep_\varnothing(G) \ar[d]^-{\eqref{eqn:Res-0}} \\
 \Rep_\varnothing(G_1 T) \ar[r]^-{\Theta_s} & \Rep_\varnothing(G_1 T).
 }
\]

\begin{lem}
\label{lem:Theta-hT}
For any $w \in \Wext$ and
any $s \in \Saff$, in the Grothendieck group of $\Rep_\varnothing(G_1 T)$ we have
 \[
  [\Theta_s(\hZ(w \cdot_p 0))] = [\hZ(w \cdot_p 0)] + [\hZ(ws \cdot_p 0)].
 \]
\end{lem}

\begin{proof}
 This follows from~\cite[Equations (2) and (3) in~\S II.9.22]{jantzen}.
\end{proof}

\begin{rmk}
 The formula in Lemma~\ref{lem:Theta-hT} suggests that the Grothendieck group $[\Rep_\varnothing(G_1 T)]$ is closely related with the right $\Z[\Waff]$-module $\Z \otimes_{\Z[v^{\pm 1}]} \Per$. However, two important remarks are in order. First, since $\mathscr{A}$ is in bijection with $\Waff$ rather than $\Wext$, to make this precise we would have to work with the ``true'' principal block $\Rep_0(G_1 T)$ in $\Rep(G_1 T)$, i.e.~the Serre subcategory generated by simple modules $\hSim(w \cdot_p 0)$ with $w \in \Waff$. But even then a difficulty would remain, since the classes of baby Verma modules do \emph{not} form a basis of the Grothendieck group $[\Rep_0(G_1 T)]$. We will not try to address this problem here.
\end{rmk}

\subsection{Injective/projective $G_1 T$-modules}
\label{ss:projective}

For $\lambda \in \bX$, we will denote by $\hQ(\lambda)$ the injective hull of $\hSim(\lambda)$ as a $G_1 T$-module, see~\cite[\S II.11.3]{jantzen}. As explained in~\cite[Equation~(3) in~\S II.11.5]{jantzen}, $\hQ(\lambda)$ is also the projective cover of $\hSim(\lambda)$ in this category. As for simple and baby Verma modules, for $\lambda,\mu \in \bX$ we have
\begin{equation}
\label{eqn:translation-proj}
 \hQ(\lambda + p\mu) \cong \hQ(\lambda) \otimes \bk_{\dot{T}}(\mu).
\end{equation}

If $\lambda \in \Wext \cdot_p 0$ then $\hQ(\lambda)$ belongs to $\Rep_\varnothing(G_1 T)$.
By~\cite[Proposition~II.11.4]{jantzen}, this module admits a filtration with subquotients of the form $\hZ(\mu)$ with $\mu \in \Wext \cdot_p 0$; moreover the number of occurrences of $\hZ(\mu)$ does not depend on the choice of such a filtration, and is equal to the multiplicity $[\hZ(\mu):\hSim(\lambda)]$. More generally, any projective object $\hQ$ in $\Rep_\varnothing(G_1 T)$ admits a filtration with subquotients of the form $\hZ(\mu)$ with $\mu \in \Wext \cdot_p 0$, and the number of occurrences of $\hZ(\mu)$ does not depend on the choice of filtration; this number will be denoted $(\hQ : \hZ(\mu))$.

\begin{lem}
\label{lem:T-Q-rho+Afund}
 For any $\omega \in \Omega$, the projective $G_1 T$-module $\hQ(t_\varsigma \wf \omega \cdot_p 0)$
 is the image under~\eqref{eqn:Res-0} of $\Til(p\varsigma + \omega \cdot_p 0)$. Moreover, for any $\mu$ in $\bX$ we have
 \[
  (\hQ(t_\varsigma \wf \omega \cdot_p 0) : \hZ(\mu)) = \begin{cases}
                                               1 & \text{if $\mu=t_\varsigma x \omega \cdot_p 0$ for some $x \in \Wf$;} \\
                                               0 & \text{otherwise.}
                                              \end{cases}
 \]
\end{lem}

\begin{proof}
The element $t_\varsigma \wf \omega \cdot_p 0 = (t_\varsigma \wf \omega_\varsigma^{-1}) \cdot_p (\omega_\varsigma \omega \cdot_p 0)$ belongs to the unique alcove which contains $(p-1)\varsigma = (t_\varsigma \wf \omega_\varsigma^{-1}) \cdot_p (\omega_\varsigma \cdot_p (-\varsigma))$ in its upper closure. Therefore, by~\cite[\S II.11.10]{jantzen}, the indecomposable projective $G_1 T$-module $\hQ(t_\varsigma \wf \omega \cdot_p 0)$ is obtained by translating from $\omega_\varsigma \cdot_p (-\varsigma)$ to $\omega_\varsigma \omega \cdot_p 0$ the $G_1 T$-module $\hQ((p-1)\varsigma)$. Moreover, by~\cite[\S II.9.16 and equation~(1) in~\S II.11.9]{jantzen} we have $\hQ((p-1)\varsigma) \cong \hSim((p-1)\varsigma) \cong \hZ((p-1)\varsigma)$. Then the description of $(\hQ(t_\varsigma \wf \omega \cdot_p 0) : \hZ(\mu))$ follows from the considerations surrounding Equations~(2)--(3) in~\cite[\S II.9.22]{jantzen}.

The first claim follows from the considerations above and Lemma~\ref{lem:Til-rho+Afund}.
\end{proof}

The following result is an easy consequence of~\cite[Corollar~4.5]{jantzen-darstellungen} (see also~\cite[\S 11.11]{jantzen}); see~\cite{donkin} or~\cite[\S E.9]{jantzen} for details.

\begin{thm}
\label{thm:jantzen}
Assume that $p \geq 2h-2$. Then for any restricted dominant weight $\lambda$ we have
\[
\hQ(\lambda) \cong \Til(2(p-1)\rho + \wf\lambda)_{| G_1T}.
\]
\end{thm}

We will say that an element $w \in \Wext$ is \emph{restricted} if $w \cdot_p 0$ is a restricted dominant weight. Note that this condition does not depend on $p$, and that restricted elements belong to $\fWext$.

In terms of the orbit $\Wext \cdot_p 0$, since $\wf(\varsigma)=\varsigma-2\rho$,
Theorem~\ref{thm:jantzen} implies in particular that (if $p \geq 2h-2$) for any $w \in \Wext$ such that $t_\varsigma w$ is restricted, we have
\begin{equation}
\label{eqn:Til-hQ}
\hQ(t_\varsigma w \cdot_p 0) \cong \Til(t_\varsigma \wf w \cdot_p 0)_{|G_1T}.
\end{equation}

\subsection{Characters of tilting modules as $G_1 T$-modules}

Now we set
\[
 \MMsphext := \Z \otimes_{\Z[v^{\pm 1}]} \Msphext,
\]
and still denote by $\varphi : \MMsphext \to \MMasphext$ and $\zeta : \MMsphext \to \Z[\Wext]$ the (injective) morphisms induced by~\eqref{eqn:def-phi} and~\eqref{eqn:zeta} respectively. We then consider the maps
\[
\xymatrix{
 \MMsphext \ar[d]_-{\zeta} \ar[r]^-{\varphi} & \MMasphext \ar[r]^-{\eqref{eqn:Groth-group}}_-{\sim} & [\Rep_\varnothing(G)] \ar[r] & [\Rep_\varnothing(G_1 T)] \\
 \Z[\Wext]
 }
\]
where the rightmost arrow is induced by the restriction functor~\eqref{eqn:Res-0}.

\begin{prop}
\label{prop:mult-tilting-babyVerma}
 Let $M$ be a tilting module in $\Rep_\varnothing(G)$, all of whose direct summands are of the form $\Til(t_\varsigma w \cdot_p 0)$ with $w \in \fWext$. Then $M_{|G_1 T}$ is a projective $G_1 T$-module. Moreover, the inverse image $a$ of $[M]$ under~\eqref{eqn:Groth-group} belongs to the image of $\varphi$, and the image under $\zeta$ of the preimage of $a$ is equal to
 \[
  \sum_{w \in \Wext} \bigl( M_{|G_1 T} : \hZ (w \cdot_p 0 + p\varsigma) \bigr) \cdot w.
 \]
\end{prop}

\begin{proof}
 As explained in~\cite[Lemma~E.8]{jantzen}, an indecomposable tilting
 module $\Til(\lambda)$ (with $\lambda \in \bX^+$) is projective as a
 $G_1 T$-module iff $\lambda - (p-1)\varsigma \in \bX^+$. This implies
 the first claim in the proposition, and also that the $G$-modules $M$
 as in the statement are all isomorphic to direct sums of direct summands of modules of the form
 \[
  \Theta_{s_1} \Theta_{s_2} \cdots \Theta_{s_n} (\Til(p\varsigma + \omega \cdot_p 0))
 \]
with $s_1, \cdots, s_n$ in $\Saff$ and $\omega \in \Omega$. This reduces the proof of the proposition to the case of modules of this form. We will prove this case by induction on $n$.

First we treat the case $n=0$. By~\eqref{eqn:uN-varsigma-2}
and Lemma~\ref{lem:Til-rho+Afund} we have
\[
[\Til(p\varsigma + \omega \cdot_p 0)] = 
1 \otimes \uN_{t_\varsigma \omega} = \varphi(1 \otimes M_\omega).
\]
(Of course, this equality is also a special case of Theorem~\ref{thm:amrw}.)
Using~\eqref{eqn:zeta-id}, we deduce that the image under $\zeta$ of the preimage of $[\Til(p\varsigma + \omega \cdot_p 0)]$ is
\[
\sum_{x \in \Wf} 1 \otimes H_{x\omega}.
\]
On the other hand, by Lemma~\ref{lem:T-Q-rho+Afund} we know that $\Til(p\varsigma + \omega \cdot_p 0)_{| G_1T} = \hQ(t_\varsigma \wf \omega \cdot_p 0)$, and we know the multiplicities of baby Verma modules in this projective module. Comparing with the formula above, we deduce the desired claim.


To prove the induction step, we will prove that if the claim is true
for a module $M$, then it is true also for $\Theta_s(M)$ for any $s \in S$. As explained just after~\eqref{eqn:Groth-group}, if we denote by $a$ the inverse image of $[M]$, then the inverse image of $[\Theta_s(M)]$ is $a \cdot (\id + s)$. Hence if $a=\varphi(b)$, then this inverse image is $\varphi(a \cdot (\id + s))$. Now by Lemma~\ref{lem:Theta-hT}, for any $w \in \Wext$ we have
\[
 (\Theta_s(M_{|G_1 T}) : \hZ(p \varsigma + w \cdot_p 0)) = (M_{|G_1 T} : \hZ(p\varsigma + w \cdot_p 0)) + (M_{|G_1 T} : \hZ(p\varsigma + ws \cdot_p 0)),
\]
and the desired claim follows.
\end{proof}

\subsection{The simple character formula}

Our main application of Theorem~\ref{thm:main} is the following claim.

\begin{thm}
\label{thm:multiplicities}
 Assume that $p \geq 2h-2$. If $w \in \Wext$ is such that $t_\varsigma w$ belongs to $\fWext$ and is restricted, 
 then for any $y \in \Wext$ we have
 \[
  \left( \hQ(w \cdot_p 0) : \hZ(y \cdot_p 0) \right) = {}^p \hspace{-1pt} h_{y, w}(1).
 \]
\end{thm}

\begin{proof}
By~\eqref{eqn:Til-hQ}, we have
 \[
\Til(t_\varsigma \wf w \cdot_p 0)_{|G_1 T} \cong \hQ(p\varsigma + w \cdot_p 0).
 \]
 By Theorem~\ref{thm:amrw}, the class $[\Til(t_\varsigma \wf w \cdot_p 0)]$ in $[\Rep_0(G)]$ is the image of $1 \otimes \puN_{t_\varsigma \wf w}$ under~\eqref{eqn:Groth-group}. Now by Theorem~\ref{thm:main} we have
 \[
  1 \otimes \puN_{t_\varsigma \wf w} = \varphi(1 \otimes \puM_{\wf w}).
 \]
 Using~\eqref{eqn:zeta-puM} and Proposition~\ref{prop:mult-tilting-babyVerma} we deduce that
\[
 \sum_{x \in \Wext} \left( \hQ(p\varsigma + w \cdot_p 0) : \hZ(p\varsigma + x \cdot_p 0) \right) \cdot x = 1 \otimes \puH_{w}.
\]
Since $( \hQ(p\varsigma + w \cdot_p 0) : \hZ(p\varsigma + x \cdot_p 0) )= ( \hQ(w \cdot_p 0) : \hZ(x \cdot_p 0) )$ for any $x \in \Wext$ (see~\eqref{eqn:translation-babyVerma} and~\eqref{eqn:translation-proj}), this implies the desired equality.
\end{proof}

\begin{rmk}
 \begin{enumerate}
  \item
  Let $w_{\max} \in \Waff$ be the unique element such that $w_{\max} \cdot_p 0$ belongs to the (shifted and dilated) alcove of $-p\varsigma$.
  The main result of~\cite{fiebig-sheaves} (see in
  particular~\cite[Theorems~7.8 and~8.6]{fiebig-sheaves}) states that
  if $\puH_x=\uH_x$ for any $x \in \Waff$ such that $x \leq w_{\max}$
  in the Bruhat order, then we have $[\hZ(y \cdot_p 0) : \hSim(w
  \cdot_p 0)] = h_{y,w}(1)$ for any $w,y \in \Waff$ such that
  $t_\varsigma w$ belongs to $\fWext$ and is restricted (and hence
  Lusztig's conjecture holds). Of course, this claim also follows from Theorem~\ref{thm:multiplicities}.
  \item
  Once the multiplicities $( \hQ(w \cdot_p 0) : \hZ(y \cdot_p 0) )$ are known for any $w,y$ as in Theorem~\ref{thm:multiplicities}, using~\eqref{eqn:translation-babyVerma} and~\eqref{eqn:translation-proj} we can deduce these multiplicities for any $w,y \in \Waff$. Then, using the reciprocity formula
\[
 ( \hQ(w \cdot_p 0) : \hZ(y \cdot_p 0) ) = [\hZ(y \cdot_p 0) : \hSim(w \cdot_p 0)]
\]
(see~\S\ref{ss:projective})
one can deduce the multiplicities on the right-hand side of this equality. And this information allows to compute the characters of the modules $\hSim(w \cdot_p 0)$ for any $w \in \Wext$. In fact, using~\eqref{eqn:translation-simples} it suffices to do so when $w$ 
is restricted. In this case $\hSim(w \cdot_p 0)$ is the restriction of a $G$-module; hence its weights (and their multiplicities) are stable under $\Wf$. As a consequence, to determine them it suffices to compute the dominant weights appearing in $\hSim(w \cdot_p 0)$ and their multiplicities. Using the determination of the multiplicities $[\hZ(y \cdot_p 0) : \hSim(w \cdot_p 0)]$ and the ``triangularity'' of these numbers (see~\cite[Corollary~9.15(a)]{jantzen}) one can write
\begin{equation}
\label{eqn:character-simples}
 [\hSim(w \cdot_p 0)] = \sum_{x \in \Wext^{(1)}} a_{x} \cdot [\hZ(x \cdot_p 0)] + \sum_{y \in \Wext^{(2)}} b_y \cdot [\hSim(y \cdot_p 0)]
\end{equation}
with $\Wext^{(1)}, \Wext^{(2)}$ subsets of $\Wext$, in such a way that $\hSim(y \cdot_p 0)$ does not admit any dominant weight for $y \in \Wext^{(2)}$. The characters of the baby Verma modules are easy to compute, see e.g.~\cite[\S 3.1]{fiebig-moment-graph}. Hence from~\eqref{eqn:character-simples} one can compute the dominant weights appearing in $\hSim(w \cdot_p 0)$ and their multiplicities. (See also~\cite[\S4]{sobaje} for a different presentation of this procedure.)
\item Our assumptions on $p$ in Theorem~\ref{thm:multiplicities} are that $p \geq 2h-2$ and $p>h$. It is easily seen that these two conditions are equivalent to the condition that $p \geq 2h-1$.
 \end{enumerate}
\end{rmk}

\subsection{Proof of Theorem~\ref{thm:main-intro}}

We conclude the paper by explaining how Theorem~\ref{thm:multiplicities} implies Theorem~\ref{thm:main-intro} from the introduction.
As in~\S\ref{ss:intro-statement}, for $m$ in $\Per$ we denote by $[m]_{v \mapsto 1}$ its image in $\Z \otimes_{\Z[v^{\pm 1}]} \Per \cong \Z[\mathscr{A}]$. Recall that $\Waff$ acts on $\mathscr{A}$ on the right, see~\S\ref{ss:intro-sph-asph}. This induces in the natural way a structure of right $\Z[W]$-module on $\Z[\mathscr{A}]$, and it is clear from~\eqref{eqn:action-Per} that this action coincides with the one induced by the $\Haff$-action on $\Per$ (via the canonical isomorphism $\Z \otimes_{\Z[v^{\pm1}]} \Haff \cong \Z[W]$).

\begin{proof}[Proof of Theorem~\ref{thm:main-intro}]
 Since both $q_A$ and $\puP_{\hat{A}}$ are invariant under the replacement of $A$ by $\mu+A$ for $\mu \in \bX$ (see~\eqref{eqn:translation-proj} and~\eqref{eqn:puP-trans} respectively), we can assume that $A \subset \check{\Pi}_0$, so that $\hat{A}=\wf(A) \in \hat{\Pi}_0$. Then
 \[
  \puP_{\hat{A}} = \frac{1}{\pi_{\mathrm{f}}} \uP_{\Afund} \cdot \zeta(\puM_w),
 \]
where $w \in \Waff$ is the unique element such that $\wf(A)=\Afund \cdot w = w(\Afund)$, i.e.~such that $A=\wf w(\Afund)$. By construction we have
\[
 \zeta(\puM_w) = \puH_{\wf w} = \sum_{y \in \Waff} {}^p \hspace{-1pt} h_{y,\wf w} \cdot H_y,
\]
hence using~\eqref{eq:uPAfund} we obtain that
\begin{multline*}
 [\puP_{\hat{A}}]_{v \mapsto 1} = \frac{1}{|\Wf|} \left( \sum_{x \in \Wf} x(\Afund) \right) \cdot \left( \sum_{y\in\Waff} {}^p \hspace{-1pt} h_{y,\wf w}(1) \cdot y \right) \\
 = \frac{1}{|\Wf|} \left( \sum_{\substack{x \in \Wf \\ y\in\Waff}} {}^p \hspace{-1pt} h_{y,\wf w}(1) \cdot xy(\Afund) \right) \\
 = \sum_{z \in \Waff} \frac{1}{|\Wf|} \left(\sum_{x \in \Wf} {}^p \hspace{-1pt} h_{x^{-1}z,\wf w}(1) \right) \cdot z(\Afund).
\end{multline*}
Now since $w \in \fW$, the element $\wf w$ is maximal in $\Wf \cdot \wf w$, so that the parity complex $\cE_{\wf w}$ (see~\S\ref{ss:categorification}) is constructible with respect to the stratification by $\Gv_\scO$-orbits, which implies that ${}^p \hspace{-1pt} h_{x^{-1}z,\wf w}(1) = {}^p \hspace{-1pt} h_{z,\wf w}(1)$ for any $x \in \Wf$. We deduce that
\[
 [\puP_{\hat{A}}]_{v \mapsto 1} = \sum_{z \in \Waff} {}^p \hspace{-1pt} h_{z,\wf w}(1) \cdot z(\Afund).
\]
Comparing with Theorem~\ref{thm:multiplicities} and the definition of $q_A$, we obtain the desired formula.
\end{proof}

\section{A combinatorial proof of Theorem~\ref{thm:intro-phi} in the case of Kazhdan--Lusztig bases}
\label{sec:comb-proof}

In this section we provide an alternative proof of the version of Theorem~\ref{thm:intro-phi} for ``standard'' Kazhdan--Lusztig bases, see Remark~\ref{rmk:intro-Hext}\eqref{it:thm-classical}. This proof is
based on the results of~\cite{soergel-comb-tilting},\footnote{Some of
  the results of~\cite{soergel-comb-tilting} on which this proof is
  based are not originally due to Soergel, but are restatements of
  results due to Lusztig and to Kato. We refer
  to~\cite{soergel-comb-tilting} for a discussion of the original
  references.} (and is therefore ``combinatorial'') and was explained
to us by Soergel.

In this section we assume that $G$ is semisimple (and simply connected).
As in~\S\ref{ss:intro-statement} we consider the ``periodic module''
$\Per$ for $\cH$. As explained in~\cite[\S 4]{soergel-comb-tilting} the family $(\uP_A : A \in \scA)$ forms a $\Z[v^{\pm 1}]$-basis of a certain $\cH$-submodule $\Per^\circ \subset \Per$. By~\cite[Lemma~4.9]{soergel-comb-tilting}, the action of $\bX$ on $\Per$ considered in~\S\ref{ss:periodic-module} extends to a $\Z[v^{\pm 1}]$-linear action of the extended affine Weyl group $\Wext$ on $\Per^\circ$, for which the action of $w \in \Wext$ is denoted $\langle w \rangle : \Per^\circ \to \Per^\circ$. 
It follows from~\eqref{eqn:trans-action} that this action satisfies the following formula:
for $\omega \in \Omega$ and $w \in W$ we have
\[
\langle \omega w \rangle (P \cdot h)=(\langle \omega w \rangle P) \cdot (H_\omega h H_\omega^{-1})
\] 
for $P \in \Per^\circ$ and $h \in \cH$. (Here the element $H_\omega$ belongs to the larger algebra $\Hext$ introduced in~\S\ref{ss:sph-asph}; conjugation by this element stabilizes $\Haff$.)

We next consider the map
\[
\mathrm{alt} : \Per^\circ \to \Per^\circ
\]
defined by the formula
\[
\mathrm{alt}(P) = \sum_{w \in \Wf} (-1)^{\ell(w)} \langle w \rangle P.
\]
We will also consider the $\Z[v^{\pm 1}]$-linear map
\[
\mathrm{res} : \Per \to \Masph
\]
determined by
\[
\mathrm{res}(A)=\begin{cases}
N_A & \text{if $A \in \scA^+$;} \\
0 & \text{otherwise.}
\end{cases}
\]
This morphism is not $\cH$-linear; but it follows from~\cite[Proposition~5.2]{soergel-comb-tilting} that the composition $\mathrm{res} \circ \mathrm{alt}$ \emph{is} $\cH$-linear.
Moreover, \cite[Theorem~5.3(1)]{soergel-comb-tilting} says that for $A \in \scA_\rho^+$ we have
\begin{equation}
\label{eqn:res-Masph}
\mathrm{res} \circ \mathrm{alt}(\uP_A)=\uN_A.
\end{equation}

On the other hand, let us consider the morphism of right $\cH$-modules 
\[
\mathrm{Alt} := \langle t_{-\rho} \rangle \circ \mathrm{alt} \circ \langle t_\rho \rangle : \Per^\circ \to \Per^\circ.
\]
In~\cite[\S 6]{soergel-comb-tilting} Soergel introduces the $\Z[v^{\pm 1}]$-module $\hat{\Per}$ consisting of certain formal linear combinations $\sum_A f_A A$, and the map $\eta : \hat{\Per} \to \hat{\Per}$. The module $\hat{\Per}$ contains $\Per$ as a submodule in the natural way. We also have a map
\[
\mathrm{Res} : \hat{\Per} \to \Msph
\]
sending $\sum_A f_A A$ to $\sum_{A \in \scA^+} f_A M_A$. (Here the linear combination $\sum_{A \in \scA^+} f_A M_A$ is finite due to the form of the combinations authorized in $\hat{\Per}$.) Then~\cite[Corollary~6.9]{soergel-comb-tilting} states that for any $A \in \scA^+$ we have
\begin{equation}
\label{eqn:Res-Msph}
\uM_A = \mathrm{Res} \circ \eta \circ \mathrm{Alt} (\uP_A).
\end{equation}
By~\cite[Proposition~6.6]{soergel-comb-tilting}, the map $\mathrm{Res} \circ \eta \circ \mathrm{Alt}$
  is $\cH$-linear; it follows that $\mathrm{Res} \circ \eta : \mathrm{Alt}(\Per^\circ) \to \Msph$ is $\cH$-linear as well.
One can check that the elements $(\mathrm{Alt} (\uP_A) : A \in \scA^+)$ form a $\Z[v^{\pm 1}]$-basis of the sub-$\cH$-module $\mathrm{Alt}(\Per^\circ) \subset \Per^\circ$; therefore this formula implies that $\mathrm{Res} \circ \eta$ induces an isomorphism of $\cH$-modules from $\mathrm{Alt}(\Per^\circ)$ to $\Msph$.

By~\eqref{eq:Ptrans}
we have $\langle t_\rho \rangle(\uP_A)=\uP_{\rho + A}$ for any $A \in \scA$; hence the formula~\eqref{eqn:Res-Msph} can be written as
\[
\uM_A = \mathrm{Res} \circ \eta \circ \langle t_{-\rho} \rangle \circ \mathrm{alt} (\uP_{\rho+A})
\]
for any $A \in \scA^+$.

Fix now $A \in \scA^+$, and
choose $h \in \cH$ such that $\uM_A=\uM_{\Afund} \cdot h$. Then we have
\begin{multline*}
\mathrm{Res} \circ \eta \circ \langle t_{-\rho} \rangle \circ \mathrm{alt} (\uP_{\rho+A}) = \uM_A = \uM_{\Afund} \cdot h \\
= \bigl( \mathrm{Res} \circ \eta \circ \langle t_{-\rho} \rangle \circ \mathrm{alt} (\uP_{\rho+\Afund}) \bigr) \cdot h = \mathrm{Res} \circ \eta \circ \langle t_{-\rho} \rangle \circ \mathrm{alt} (\uP_{\rho+\Afund} \cdot H_{\omega_\rho} h H_{\omega_\rho}^{-1}).
\end{multline*}
By injectivity of $\mathrm{Res} \circ \eta$ on $\mathrm{Alt}(\Per^\circ)$ we deduce that
\[
\langle t_{-\rho} \rangle \circ \mathrm{alt} (\uP_{\rho+A}) = \langle t_{-\rho} \rangle \circ \mathrm{alt} (\uP_{\rho+\Afund} \cdot H_{\omega_\rho} h H_{\omega_\rho}^{-1} ),
\]
and then that
\[
\mathrm{alt} (\uP_{\rho+A}) = \mathrm{alt} (\uP_{\rho+\Afund} \cdot H_{\omega_\rho} h H_{\omega_\rho}^{-1}).
\]
Applying $\mathrm{res}$ and using~\eqref{eqn:res-Masph}, we deduce that
\begin{multline}
\label{eqn:comb-proof}
\uN_{\rho + A} = \mathrm{res} \circ \mathrm{alt} (\uP_{\rho+A}) = \mathrm{res} \circ \mathrm{alt} (\uP_{\rho+\Afund} \cdot H_{\omega_\rho} h H_{\omega_\rho}^{-1}) \\
= \bigl( \mathrm{res} \circ \mathrm{alt} (\uP_{\rho+\Afund}) \bigr) \cdot H_{\omega_\rho} h H_{\omega_\rho}^{-1} = \uN_{\rho+\Afund} \cdot H_{\omega_\rho} h H_{\omega_\rho}^{-1}.
\end{multline}

Now we have $S_\rho=\tau_\rho(\Sf)=\omega_\rho \Sf \omega_\rho$, see Remark~\ref{rmk:Omega}. Hence there exists an isomorphism of $\Z[v^{\pm 1}]$-modules
\[
\Msph \simto \Msph_{\overline{\rho}}
\]
which sends $M_\id \cdot h$ to $M^{\overline{\rho}}_\id \cdot H_{\omega_\rho} h H_{\omega_\rho}^{-1}$ for any $h \in \cH$. In terms of the parametrization by alcoves, this morphism sends $M_B$ to $M^\rho_{\rho+B}$ for any $B \in \scA^+$. (In fact, if $B=w(\Afund)$ with $w \in \fW$, then $M_B=M_{\Afund} \cdot H_w$ is sent to $M^\rho_{\rho+\Afund} \cdot H_{\omega_\rho} H_w H_{\omega_\rho}^{-1}=M^\rho_{(\rho+\Afund) \cdot \omega_\rho w \omega_\rho^{-1}}=M^\rho_{x_\rho \omega_\rho w \omega_\rho^{-1}(\Afund)}=M^\rho_{\rho+B}$.) This morphism also commutes with the appropriate Kazhdan--Lusztig involutions, hence sends $\uM_B$ to $\uM^\rho_{\rho+B}$ for any $B \in \scA^+$. It follows that
\[
M^\rho_{\rho+\Afund} \cdot H_{\omega_\rho} h H_{\omega_\rho}^{-1}=\uM^\rho_{\rho+A},
\]
which proves that
\[
\uN_{\rho+\Afund} \cdot H_{\omega_\rho} h H_{\omega_\rho}^{-1} = \varphi_\rho(\uM^\rho_{\rho+A}).
\]
Comparing with~\eqref{eqn:comb-proof} we finally obtain that $\uN_{\rho + A}=\varphi_\rho(\uM^\rho_{\rho+A})$, which finishes the proof.


\end{document}